\documentclass[12pt,leqno]{amsart}
\usepackage{amsmath,amsthm,amssymb,amscd,url}
\usepackage{enumitem}
\usepackage[colorlinks=true]{hyperref}
\usepackage[all]{xy}

\newcommand{\TITLE}{Integral points on elliptic curves and explicit valuations of division polynomials}
\newcommand{\TITLERUNNING}{Integral points on elliptic curves and explicit valuations of division polynomials}
\newcommand{\DATE}{\today}

\theoremstyle{plain} 
\newtheorem{theorem}{Theorem} 

\newtheorem*{langconjecture}{Lang's Height Conjecture}
\newtheorem*{langconjecturetwo}{Hall-Lang Conjecture}

\newtheorem{proposition}[theorem]{Proposition}
\newtheorem{lemma}[theorem]{Lemma}
\newtheorem{corollary}[theorem]{Corollary}

\theoremstyle{definition}
\newtheorem{definition}[theorem]{Definition}

\theoremstyle{remark}
\newtheorem{remark}[theorem]{Remark}
\newtheorem{example}[theorem]{Example}

\newtheorem*{acknowledgement}{Acknowledgements}

  {\end{list}}

  {\end{list}}
\newcommand{\CC}{\mathbb{C}}
\newcommand{\FF}{\mathbb{F}}

\newcommand{\QQ}{\mathbb{Q}}

\newcommand{\ZZ}{\mathbb{Z}}

\renewcommand{\gcd}{{\operatorname{gcd}}}

\newcommand{\MOD}[1]{~(\textup{mod}~#1)}
\renewcommand{\pmod}{\MOD}

\newcommand{\ord}{\operatorname{ord}}

\renewcommand{\setminus}{\smallsetminus}

\newcommand\al{\left\lfloor \frac{a}{\ell} \right\rfloor}
\newcommand\nal{\left\lfloor \frac{na}{\ell} \right\rfloor}

\newcommand\Sil{MR2514094}
\newcommand\Siltwo{MR1312368}
\newcommand\Ingram{MR2468477}
\newcommand\CheonHahn{MR1654780}
\newcommand\CheHahTwo{MR1683630}
\newcommand\AyadS{MR1185022}

\title[\TITLERUNNING]{\TITLE}
\date{\DATE}
\author{Katherine E. Stange} 
\address{%
Department of Mathematics,
University of Colorado Boulder,
Campus Box 395,
Boulder, CO, 80309, USA}
\email{kstange@math.colorado.edu}

\thanks{The author's research has been supported by NSERC PDF-373333 and NSF MSPRF 080291.}

\begin{document}
%%%%%%%%%%%%%%%%%%%%%%%%%%%%%%%%%%%%%%%%%%%%%%%%%%%%%%%%%%%%%%%%%%%%%%
%%% Text (non-TeX) Abstract
%%% Assuming Lang's conjectured lower bound on the heights of non-torsion points on an elliptic curve, we show that there exists an absolute constant C such that for any elliptic curve E/Q and non-torsion point P in E(Q), there is at most one integral multiple [n]P such that n > C.  The proof is a modification of a proof of Ingram giving an unconditional but not uniform bound.  The new ingredient is a collection of explicit formulae for the sequence of valuations of the division polynomials.  For P of non-singular reduction, such sequences are already well described in most cases, but for P of singular reduction, we are led to define a new class of sequences called elliptic troublemaker sequences, which measure the failure of the Neron local height to be quadratic.  As a corollary in the spirit of a conjecture of Lang and Hall, we obtain a uniform upper bound on h(P)/h(E) for integer points having two large integral multiples.
%%%%%%%%%%%%%%%%%%%%%%%%%%%%%%%%%%%%%%%%%%%%%%%%%%%%%%%%%%%%%%%%%%%%%%

\begin{abstract}
  Assuming Lang's conjectured lower bound on the heights of non-torsion points on an elliptic curve, we show that there exists an absolute constant $C$ such that for any elliptic curve $E/\QQ$ and non-torsion point $P \in E(\QQ)$, there is at most one integral multiple $[n]P$ such that $n > C$.  The proof is a modification of a proof of Ingram giving an unconditional but not uniform bound.  The new ingredient is a collection of explicit formul{\ae} for the sequence $v(\Psi_n)$ of valuations of the division polynomials.  For $P$ of non-singular reduction, such sequences are already well described in most cases, but for $P$ of singular reduction, we are led to define a new class of sequences called \emph{elliptic troublemaker sequences}, which measure the failure of the N\'eron local height to be quadratic.  As a corollary in the spirit of a conjecture of Lang and Hall, we obtain a uniform upper bound on $\widehat{h}(P)/h(E)$ for integer points having two large integral multiples.
\end{abstract}

\maketitle

\section{Introduction}

A famous theorem of Siegel states that there are only finitely many integral points on any elliptic curve $E/\QQ$.  Of course, this implies that among the multiples $[n]P$ of any particular point $P$, only finitely many may be integral.  In this context there are two natural ways to give a bound:  on the number of such points; and on the size of $n$.  If one assumes either the \emph{abc} Conjecture of Masser and Oesterl\'e, or Szpiro's Conjecture, and restricts attention only to elliptic curves in minimal Weierstrass form, then the \emph{number} of integral points among the multiples of $P$ is bounded uniformly according to work of Hindry and Silverman \cite{MR948108}.  This is also known unconditionally for curves of integral $j$-invariant \cite{MR895285}.  

The best known result bounding the \emph{size} of $n$ is due to Ingram \cite{\Ingram}, who uses lower bounds on linear forms in elliptic logarithms to bound $n$ in terms of the height of $E$, and the quantity $M(P)$, defined as the smallest $m$ such that $[m]P$ has non-singular reduction modulo all primes (note that $M(P)$ can be bounded above in terms of $E$ alone).  Using a gap principle, Ingram goes on to find a constant $C$, depending only on $M(P)$ (and not the height of $E$) such that at most one multiple $[n]P$ is integral for $n > C$.  At the moment, analogous results bounding integral points among linear combinations $[n]P + [m]Q$ seem to be out of reach.

In this paper, we obtain a similar result in which the constant depends only on the ratio of heights $h(E)/\widehat{h}(P)$, defined as follows.
The canonical height of a point $P$ is given by
\[
\widehat{h}(P) := \frac{1}{2} \lim_{n \rightarrow \infty} \frac{h([2^n]P)}{4^n},
\]
where $h(P) := h(x(P))$, the logarithmic height of the $x$-coordinate.  The height of $E$ is
  \[
h(E) := \max\{ h(j), \log|\Delta|, 1 \}.
  \]
See Section \ref{sec:integral} for more detail.

\begin{theorem}
  There are uniform constants $C$ and $C'$ such that for all elliptic curves $E/\QQ$ in minimal Weierstrass form, and non-torsion points $P \in E(\QQ)$, there is at most one value of 
  \begin{equation}
          \label{eqn:thm1}
  n > \max \left\{ C \frac{h(E)}{\widehat{h}(P)} \log\left( \frac{h(E)}{\widehat{h}(P)} \right), C' \right\}
  \end{equation}
  such that $[n]P$ is integral.  Furthermore, this one value is prime.
  \label{thm:patrick-new}
\end{theorem}

The bulk of the proof consists of giving complete closed formul{\ae} for the valuations of elliptic divisibility sequences at primes of bad reduction, which ingredient is combined with established methods of Ingram \cite{\Ingram}.  The formul{\ae} themselves are considered a principle goal of this paper, but we will first discuss the implications of Theorem \ref{thm:patrick-new}.

The restriction to minimal elliptic curves is necessary, since otherwise there exist methods of constructing examples with arbitrarily many integral points.  Throughout the paper we compare this result to Theorem 1 of Ingram \cite{\Ingram}, which differs only in that the bound \eqref{eqn:thm1} is replaced by $n > CM(P)^{16}$.  These bounds are quite different: for example, one does not expect curves of large height necessarily to have large $M(P)$.  The quantity $M(P)$ divides the least common multiple of the Tamagawa numbers of $E$, which measure the number of components in the fibres of the N\'eron model.

Unfortunately, the constants $C$ and $C'$ in Theorem \ref{thm:patrick-new}, while effective, are quite large.  More details can be found in Section \ref{sec:integral}.

The bound becomes uniform if one assumes a well-known conjecture of Lang, here given in a slightly strengthened form (for details, see Section \ref{sec:integral}). 

\begin{langconjecture}[{\cite[p. 92]{MR518817}, \cite[Conjecture 9.9]{\Sil}}]
  There is a uniform constant $C_L$ such that for any elliptic curve $E/\QQ$ in minimal Weierstrass form, and point $P \in E(\QQ)$ of infinite order, 
  \[
  \widehat{h}(P) > C_L h(E).
  \]
\end{langconjecture}

Lang's conjecture follows from the \emph{abc} Conjecture, via Szpiro's Conjecture \cite{MR948108}.  The bound on $n$ in Theorem \ref{thm:patrick-new} becomes uniform if we assume any of these conjectures.  In particular, the bound is uniform if we restrict to elliptic curves of integral $j$-invariant, or curves for which the denominator of the $j$-invariant is divisible by a bounded number of primes, for which Lang's conjecture is known to hold \cite{MR948108, MR630588}.  The uniformity of the bound in Theorem \ref{thm:patrick-new} in the case of integral $j$-invariant is already a result of Ingram's original argument \cite{\Ingram} (but this does not extend to curves whose $j$-invariant is divisible by a bounded number of primes).

An immediate corollary to the main result is the following.

\begin{corollary}
  There are uniform constants $c$ and $c'$ such that for any elliptic curve $E/\QQ$ in minimal Weierstrass form, and point $P \in E(\QQ)$ of infinite order having at least two integral multiples $[n]P$ and $[m]P$ satisfying $n > m > c'$, then
  \[
  \widehat{h}(P) \le c h(E).
  \]
  \label{maincor}
\end{corollary}
Call a triple $(P, n, m)$ satisfying the hypotheses of Corollary \ref{maincor} \emph{far-out} for the constant $c'$.  Then the theorem and corollary state that Lang's Height Conjecture is incompatible with the existence of far-out triples for arbitrarily large $c'$; in essence, sufficiently far-out triples would generate examples of points $P$ with large ratio $h(E)/\widehat{h}(P)$ (i.e. extreme examples for Lang's Height Conjecture).

Compare to another conjecture of Hall and Lang, which posits an \emph{upper} bound on the height of an integral point in terms of the height of the curve.  Note that any $P$ with integral multiples is necessarily integral.

\begin{langconjecturetwo}[{\cite[Conjecture 5]{MR717593}}]
  There is a uniform constant $C_{HL}$ such that for any elliptic curve $E/\QQ$ in minimal Weierstrass form, and integral point $P \in E(\QQ)$ of infinite order,
  \begin{equation}
    \label{eqn:langintheight}
    \widehat{h}(P) < C_{HL} h(E).
\end{equation}
\end{langconjecturetwo}

This conjecture generalises a conjecture of Hall for elliptic curves of the form $y^2 = x^3 + b$ \cite{MR0323705}. In \cite{MR717593}, Lang used a slightly different definition of $h(E)$ than we use here; see Section \ref{sec:integral} for a justification that they are equivalent.  The conjecture seems out of reach; the best known bounds with a uniform constant are exponential in $h(E)$ \cite{MR0231783, MR1145607, MR1288309, MR0340175}, but see also \cite{MR1348477}.

In this light, Corollary \ref{maincor} states roughly that far-out triples satisfy a Hall-Lang bound (even if Lang's Height Conjecture holds, it may allow `moderately' far-out triples).  More interestingly, a strengthening of Theorem \ref{thm:patrick-new} would lead to a Hall-Lang result in cyclic subgroups.  Specifically, if \eqref{eqn:thm1} could be strengthened to
\[
        n > \max\left\{ C \left( \frac{h(E)}{\widehat{h}(P)} \right)^\frac12, C' \right\},
\]
then one would obtain the following statement:  There are uniform constants $D_1$ and $D_2$ such that for any integral point $P$, all integral multiples $Q = [n]P$ of $P$ satisfying $n>D_1$, except at most one, satisfy $h(E)/\widehat{h}(Q) > D_2$.  To derive this, one uses the fact that $h([n]P) = n^2h(P) + O(1)$ (see, for example, \cite[\S VIII.6]{\Sil}.

Theorem \ref{thm:patrick-new} is proven using the following estimate for $P$ such that $[n]P$ is integral (Proposition \ref{prop:patrick-too}):
\[
\widehat{h}(P) \le  \log n + \frac{16}{3} h(E).
\]
Ingram's argument depends upon a similar estimate, which in turn depends upon examination of the division polynomials $\Psi_n$ of an elliptic curve.  In particular, for a point $P$, he bounds the size $|\Psi_n(P)|$ in relation to the denominator $D_n$ of $[n]P$, by considering the valuations $v_p(\Psi_n(P))$ for each prime.  The sequence $W_n = \Psi_n(P)$ is called an \emph{elliptic divisibility sequence}, or EDS.

As one might expect, the arithmetic geometry of the underlying curve and point shows itself in the number theory of the elliptic divisibility sequence, which is a subject of interest in its own right.  In fact, if one pursues an analogy to the relationship between an EDS and its underlying curve, replacing the elliptic curve with a twist of the multiplicative group, then one obtains, instead of an EDS, a Lucas sequence of the first kind, such as the Mersenne or Fibonacci numbers. The centuries-old number theoretic questions about Lucas sequences, such as the prime factorisation of their terms, when asked about elliptic divisibility sequences, translate to questions about the arithmetic geometry of $P$ and $E$, such as the orders of $P$ under reduction to finite fields.  A great many of these questions have been studied for EDS:  appearance of prime terms \cite{MR1815962, MR1961589}, primitive divisors \cite{MR2220263, ingramsilverman06, MR2555700}, squares and powers \cite{MR2164113, MR2669714, Mahe-Explicit-bounds} and the sign of terms \cite{MR2226354}, to name a few.

In this paper, we give a full, explicit description of the possible sequences of valuations $v(W_n)$ for an EDS over a $p$-adic field or a number field.  Ingram's result depends on work of Cheon and Hahn on such sequences of valuations \cite{\CheonHahn}, and it is here that the dependence on $M(P)$ arises.  Cheon and Hahn describe the sequence of valuations recursively, determining a growth rate.  In contrast, this paper provides a closed form whose parameters depend on the reduction properties of $P$ and $E$.  It is this that allows us to prove the estimate of Proposition \ref{prop:patrick-too} and therefore Theorem \ref{thm:patrick-new}.  However, it is the intention of this paper to give a complete description of these valuation sequences for its own sake, and this work makes up the bulk of the paper.

Proposition \ref{prop:patrick-too} is obtained from Lemma \ref{lemma:dnwn} of this paper, which is the moment at which the EDS results are used to feed into the proof of Theorem \ref{thm:patrick-new}.  It came to the author's attention after this paper was written that the same result appears in Mahe \cite[Proposition 4.2.3]{Mahe-Explicit-bounds}, with a proof by different methods.

For primes of good reduction for the associated elliptic curve, the sequence of valuations at a prime place is well understood and has a simple, pleasing description which has become a sort of `folk theorem,' although its first appearance in print is due to Cheon and Hahn \cite{\CheonHahn, \CheHahTwo} (but see Remark \ref{remark:thm-errors}).  

\begin{theorem}[introductory form of Theorem \ref{thm:non-sing}; see Section \ref{sec:good} for references to other versions appearing in the literature]

\label{thm:goodprimes-folk}
Let $E$ be an elliptic curve with good reduction over a $p$-adic field $K$ with valuation $v$.
Let $n_P > 1$ be the order of $P \in E_0(K)/E_1(K)$.  
Suppose that $E$ is a minimal Weierstrass model and that $v(W_{n_P}) > \frac{v(p)}{p-1}$.  
Then 
  \[
  v(W_n) =  \left\{ \begin{array}{ll}
    v(W_{n_P}) + v(n/n_P) & n_P \mid n, \\
    0 & n_P \nmid n.
  \end{array} \right.
  \]
\end{theorem}

In Theorem \ref{thm:non-sing}, we give a more complete characterisation than has, to our knowledge, appeared in the literature.  In particular, we remove the assumption that $v(W_{n_P}) > \frac{v(p)}{p-1}$, at the cost of some extra complication to the formula.

In contrast to the good reduction case, the primes of bad reduction often pop up in great quantity in an EDS, in frequency depending on the reduction of $P$ on the N\'{e}ron model.  We now state an introductory theorem combining all types of reduction (each treated separately in the paper).

\begin{theorem}[introductory combination of Theorems \ref{thm:change-to-minimal}, \ref{thm:non-sing}, \ref{thm:pot-good}, \ref{thm:mult} and \ref{thm:pot-mult}]
\label{thm:main-intro}
Let $K$ be an unramified extension of $\QQ_p$, $p \neq 2$.
Let $W_n$ be an EDS, associated to an elliptic curve $E/K$ in Weierstrass form, and non-torsion $P \in E(K)$.  There exist integers $a, \ell, c_1, c_2, c_3, c_4, c_5 $
such that
\[
v(W_n) = \frac{1}{c_1} \left( 
R_n(a,\ell)
 + c_2n^2 + c_3 +  \left\{ \begin{array}{ll}
c_4 + v(n) & c_5 \mid n \\
0 & c_5 \nmid n \\
\end{array} \right. \right) .
\]
where
\begin{equation*}
R_n(a,\ell) 
=
\left\lfloor
\frac{n^2 \widehat{a} (\ell -\widehat{a})}{2\ell}
\right\rfloor 
- \left\lfloor \frac{\widehat{na}(\ell-\widehat{na})}{2\ell} \right\rfloor ,
\end{equation*}
and $\widehat{x}$ denotes the least non-negative residue of $x$ modulo $\ell$.

Furthermore, 
\[
a = 0 \iff R_n(a, \ell) \equiv 0 \iff \left\{ \begin{array}{l} E \mbox{ has potential good reduction or } \\ P\mbox{ has non-singular reduction} \end{array} \right\}.
\]
\end{theorem}

The full results in this paper apply to all $p$-adic fields and to torsion points, at the cost of some complication to the final term of the formula.  They also provide much more detail about the significance and possible values of the parameters.  The sequences $R_n(a,\ell)$, here dubbed \emph{elliptic troublemaker sequences}, satisfy a host of properties examined in Section \ref{sec:ets}.

For $P$ such that $[n]P \neq \mathcal{O}$, the valuations of $W_n$ are connected to N\'eron-Tate local heights by the following relationship \cite[Exercise VI.6.4(e)]{\Siltwo}:
\[
\lambda_v([n]P) = n^2 \lambda_v(P) - \log |W_n|_v + \frac{n^2-1}{12} \log |\Delta|_v.
\]
Some portions of the results in this paper can be viewed as results about local heights, and could be proven by recourse to the established theory of such.

In the case of good reduction in minimal Weierstrass form, Theorem \ref{thm:goodprimes-folk} implies that $v(W_n)$ is asymptotically equal to $v(n)$ as a function of $n$ (on the non-zero terms).  Cheon and Hahn show, using a recurrence relation for EDS (see \eqref{eq:w-rec}), that for $P$ non-torsion having singular reduction, $v(W_n)$ is asymptotically equal to $Cn^2$ for some constant $C$ \cite{\CheonHahn}.  Everest and Ward use the elliptic Jensen formula to give a growth rate in this situation of
\[
\log | \Psi_n(P) |_v =  \left( \lambda_v(P) + \log | \Delta |_v /12 \right) n^2 + O(n^C),
\]
for some constant $0 < C < 2$ which may depend on $P$ \cite[Theorem 3]{MR1800354}.  Here, $\lambda_v(P)$ is the N\'eron local height (note that \cite{MR1800354} uses a different normalisation than ours; we follow Silverman \cite[Chapter VII]{\Siltwo}).  Everest and Ward use this result to give an algorithm for computing the canonical height of a point.  Theorem \ref{thm:growth} improves the error term on this estimate to $O(\log n)$; see Section \ref{subsec:growth}.

For any torsion point $P$, the $W_n$ are supported only on primes of bad reduction (see Remark \ref{remark:torsionreduction}).  Gezer and Bizim give some explicit descriptions of these valuations over $\QQ$ for $N < 13$ \cite[Theorem 2.2]{GezerBizim}.  Their formul{\ae} can be restated in terms of elliptic troublemaker sequences.  See Section \ref{subsec:gezerbizim}.

Sections \ref{sec:divpoly-back} and \ref{sec:eds-back} provide background.  Section \ref{sec:form} generalises the central lemma on formal groups that lies at the core of Theorem \ref{thm:goodprimes-folk}.  Sections \ref{sec:good}, \ref{sec:pot-good}, \ref{sec:mult} and \ref{sec:pot-mult} describe the valuation sequence $v(W_n)$ for each type of reduction.  Section \ref{sec:ets} considers the properties of elliptic troublemaker sequences.  Section \ref{sec:integral} proves Theorem \ref{thm:patrick-new}, while Section \ref{sec:applications} examines a few other connections and applications.  Finally, Section \ref{sec:examples} gives some detailed examples of elliptic divisibility sequences and their sequences of valuations.

\begin{acknowledgement}
        The author would like to thank the anonymous referees, Patrick Ingram, Dino Lorenzini, Val\'ery Mahe, Joseph H. Silverman, Paul Voutier and Jonathan Wise for helpful suggestions.  The examples in this paper were computed with Sage Mathematics Software \cite{sage}; scripts for computing elliptic divisibility sequences in Sage are available on the author's website \url{http://math.colorado.edu/~kstange/}.
\end{acknowledgement}

\section{Preliminaries on division polynomials}
\label{sec:divpoly-back}

In this section, we briefly catalogue some of the standard properties of division polynomials.  The proofs are largely computational, and are omitted.  For background, see especially \cite[Section 2]{\AyadS}, but also \cite[Chapter 9]{MR2228252}, \cite[Exercise III.3.7]{\Sil}.  Division polynomials are usually defined for elliptic curves, but here we will suppose only a cubic curve $E$ (possibly singular) given in standard Weierstrass form
\[
E: y^2 + a_1xy + a_3 y = x^3 + a_2 x^2 + a_4x + a_6.
\]
The division polynomials $\Psi_n \in \ZZ[a_1, a_2, a_3, a_4, a_6, x,y]$ for the curve $E$ are defined recursively using the initial values
\begin{align*}
  \Psi_1 &= 1, \\
  \Psi_2 &= 2y + a_1x + a_3, \\
  \Psi_3 &= 3x^4 + b_2x^3 + 3b_4x^2 + 3b_6x + b_8, \\
  \Psi_4 &= \Psi_2 \cdot \left( 2x^6 + b_2x^5 + 5b_4 x^4 + 10b_6x^3 + 10b_8x^2 \right. \\
  &\quad\quad\quad\quad\quad\quad\left.+ (b_2b_8 - b_4b_6)x + (b4b_8 - b_6^2)\right), 
\end{align*}
(here $b_i$ are the usual quantities \cite[Section III.1]{\Sil}) and the recurrences
\begin{equation}
  \label{eqn:psirecs}
  \begin{aligned}
  \Psi_{2m+1} &= \Psi_{m+2}\Psi_m^3 - \Psi_{m-1}\Psi_{m+1}^3, \quad \mbox{for $m \ge 2$} \\
  \Psi_{2m}\Psi_2 &= \Psi_{m-1}^2\Psi_m\Psi_{m+2} - \Psi_{m-2}\Psi_m\Psi_{m+1}^2, \quad \mbox{for $m \ge 3$}.
\end{aligned}
\end{equation}
If $\Psi_2 = 0$, then let $\Psi_{2m}=0$ for all $m$.  For an elliptic curve $E$, the $n$-th division polynomial vanishes at all non-trivial $n$-torsion points:  it has divisor $\sum_{Q \in E[n]} (Q) - n^2 (\mathcal{O})$.  (We will use $\mathcal{O}$ for the identity of an elliptic curve.)
\begin{proposition}
  \label{prop:divtorsion}
  Let $E$ be an elliptic curve.  Then $P$ is a non-trivial $n$-torsion point if and only if $\Psi_n(E,P) = 0$.
\end{proposition}

There exist $\phi_n, \omega_n \in \ZZ[a_1, a_2, a_3, a_4, a_6, x,y]$ such that the multiplication-by-$n$ formul{\ae} for an elliptic curve are given by:
      \begin{equation*}
	\label{eqn:npprop}
      [n]P = \left( \frac{\phi_n}{\Psi_n^2}, \frac{\omega_n}{\Psi_n^3} \right).
    \end{equation*}
    In fact, $\phi_n$ and $\omega_n$ can be given by the following relations: 
    \begin{align*}
      \phi_n &= x \Psi_n^2 - \Psi_{n-1}\Psi_{n+1}, \\
      4y\omega_n &= \Psi_{n-1}^2 \Psi_{n+2} - \Psi_{n-2} \Psi_{n+1}^2.
    \end{align*}

    If we assign the natural weights 
    \begin{equation}
      \label{eqn:weights}
      w(x) = 2,\quad w(y) = 3,\quad  w(a_i) = i,
    \end{equation}
    then the Weierstrass equation is homogeneous of weight $6$.  Any change of coordinates between Weierstrass equations of the form
    \[
    x' = u^2x, \quad y' = u^3y
    \]
    changes the coefficients according to $a_i' = u^ia_i$ and $\Delta' = u^{12}\Delta$.  These weights are useful in determining the valuations of division polynomials.

\begin{proposition}
  \label{prop:divpoly}  The division polynomials $\Psi_n$ have the following properties.
  \begin{enumerate}[label=(\roman{*}), ref=(\roman{*})]

    \item \label{item:homog}  Using the natural weights \eqref{eqn:weights}, $\Psi_n$, $\phi_n$, and $\omega_n$ are homogeneous of weight $n^2-1$, $2n^2$ and $3n^2$, respectively.  
    \item \label{item:polyform}  As polynomials in $x$,
      \begin{align*}
	\Psi_n^2 &= n^2 x^{n^2-1} + (\mbox{lower order terms}) \in \ZZ[a_1, a_2, a_3, a_4, a_6, x], \\
	\phi_n &= x^{n^2} + (\mbox{lower order terms}) \in \ZZ[a_1, a_2, a_3, a_4, a_6, x].
    \end{align*}
  \item \label{item:valsdivpoly} If $E$ is given by a $v$-integral Weierstrass equation, where $v$ is a non-archimedean valuation, and $v(x), v(y) < 0$, then $v(\phi_n) = n^2v(x)$.
  \item \label{item:homothety}  The change of variables $x' = u^2x + r$ and $y' = u^3y + sx + t$ from $E$ to $E'$ gives
    \[
    \Psi_n(x',y',E') = u^{n^2-1}\Psi_n(x,y,E).
    \]
  \item \label{item:div}  Whenever $n \mid m$ for $n,m\ge 1$, we have $\Psi_n \mid \Psi_m$.
  \end{enumerate}
\end{proposition}

The division polynomials satisfy the more general recurrence equation
\begin{equation*}
  \label{eqn:wnrec}
  \Psi_{n+m+s}\Psi_{n-m}\Psi_{r+s}\Psi_r + 
  \Psi_{m+r+s}\Psi_{m-r}\Psi_{n+s}\Psi_n + 
  \Psi_{r+n+s}\Psi_{r-n}\Psi_{m+s}\Psi_m = 0,
\end{equation*}
from which the recurrences \eqref{eqn:psirecs} can be obtained as special cases \cite[Theorem 3.7]{Stange10}.

Finally, using the Weierstrass $\sigma$-function and the usual complex uniformization $\CC/\Lambda$ of an elliptic curve over $\CC$, Ward showed that \cite[Theorem 12.1]{MR0023275},
\begin{equation}
  \label{eqn:cpsi}
\Psi_n(z, \Lambda) =  \frac{\sigma(nz, \Lambda)}{\sigma(z,\Lambda)^{n^2}}.
\end{equation}

\section{Preliminaries on elliptic divisibility sequences}
\label{sec:eds-back}

\begin{definition}
  \label{defn:w}
  An \emph{elliptic divisibility sequence}, or EDS, is a sequence $W_n$ in an integral domain\footnote{One could define such sequences in a more general context, but we will be concerned only with local and global fields.} satisfying
\begin{equation}
  \label{eq:w-rec}
W_{n+m}W_{n-m}W_r^2 + W_{m+r}W_{m-r}W_n^2 + W_{r+n}W_{r-n}W_m^2 = 0.
\end{equation}
\end{definition}

The connection between EDS and elliptic curves is described by Ward in his original memoir on the subject.  We state an updated version of Ward's theorem which applies to fields of characteristic zero and cubic curves:

\begin{theorem}[{Ward \cite[Theorem 12.1]{MR0023275}, Shipsey \cite[Theorem 4.5.3]{Shipsey00}, S. \cite{Stange10}}]
  \label{thm:w}
Let $E$ be a cubic curve defined over a field $K$ of characteristic zero, given by Weiestrass form, and let $P \in E(K)$.  Then $W_n = \Psi_n(P)$ is an elliptic divisibility sequence.  
Furthermore, if $W_n \in K$ is an elliptic divisibility sequence with $W_2W_3 \neq 0$, $W_1=1$, then there exists a cubic Weierstrass curve $E/K$ and $P \in E(K)$ so that $W_n = \Psi_n(P)$.
\end{theorem}

In other words, any non-degenerate EDS over $K$ appears as the sequence of division polynomials for some cubic Weierstrass curve $E/K$ evaluated at a point $P \in E(K)$.  (As \eqref{eq:w-rec} is homogeneous, we may first scale so that $W_1=1$.)  The term `divisibility' in {`elliptic divisibility sequence'} refers to Proposition \ref{prop:divpoly}\ref{item:div}.  In particular, any EDS arising from a rational point on an elliptic curve $E/\QQ$ in minimal Weierstrass form is an integer sequence with the property that $n \mid m \implies W_n \mid W_m$.

Let $E$ be an elliptic curve defined over $K$ (by this we will always mean that $E$ is given by a Weierstrass equation), and let $\mathcal{O} \neq P \in E(K)$.  Let $W_n$ be the elliptic divisibility sequence associated to $E$ and $P$.  If we change the Weierstrass equation for $E$, we may change the elliptic divisibility sequence.  For example, it will be convenient to change the equation to one in minimal Weierstrass form, so we can consider the reduction type.  Fortunately, the associated elliptic divisibility sequence changes in a simple fashion, as described by Proposition \ref{prop:divpoly}\ref{item:homothety}.  This immediately gives the following result, which will be important enough for the later results that we include it as a theorem.

\begin{theorem}
  \label{thm:change-to-minimal}
  Let $E$ be an elliptic curve defined over a $p$-adic field $K$ with valuation $v$, given by Weierstrass form, and let $\mathcal{O} \neq P \in E(K)$.  Let $W_n$ be the associated elliptic divisibility sequence.  Then there exists an isomorphism $\phi: E \rightarrow E'$, defined over $K$, to an elliptic curve $E'$ in minimal Weierstrass form.  Let $W_n'$ be the elliptic divisibility sequence associated to $\phi(P)$. Then there exists an $r_P \in \ZZ$ such that
  \begin{equation*}
    v(W_n) = (n^2-1)r_P + v(W_n').
    \label{eqn:change-to-minimal}
  \end{equation*}
\end{theorem}

\section{Notation}
\label{sec:notation}

Throughout the remainder of the paper (except in the last three sections, \ref{sec:integral} through \ref{sec:examples}), let $p$ be a prime, let $K$ be a finite extension of $\QQ_p$, and let $R$ be the ring of integers of $K$, with maximal ideal $\mathcal{M}$.  Let $v$ be a valuation for $K$, let $\pi$ be a uniformizer, and let $\mathbf{k}$ be the residue field.  Let $E$ be an elliptic curve defined over $K$, let $P \in E(K)$, and let $W_n$ be the EDS associated to $E$ and $P$.

\section{Central Lemma on Formal groups}
\label{sec:form}

For a point of non-singular reduction, the sequence of valuations $v(W_n)$ is controlled by the formal group of the elliptic curve.  For points of singular reduction, the sequence of valuations is partially controlled by the formal group of either the elliptic curve, or the multiplicative group, depending on the type of reduction.  In both cases, the results rely on a lemma describing the valuations of the multiples of a point in an abstract formal group.  Although the formula \eqref{eqn:vals-form} below is quite complicated, in most cases we encounter, the variable $j$ takes the value $0$, whereupon \eqref{eqn:vals-form} simply reduces to $v(z) + v(n)$.  For background on formal groups, see \cite[Chapter IV]{\Sil}. 

\begin{lemma}
\label{lemma:vals-form}
Let $\mathcal{F}$ be a one-parameter formal group defined over $R$, and let $z \in \mathcal{F}(\mathcal{M})$.  
There exist integers $b$, $j$, $h$, and $w \in \ZZ^{\ge 0} \cup \{ \infty \}$ such that
for all integers $n$, 
\begin{equation}
\label{eqn:vals-form}
v([n]z) 
= \left\{
\begin{array}{ll}
  b^jv(z) + \frac{b^j-1}{b-1}h + v(n) - jv(p) + w & \mbox{if }v(n) > jv(p) \\
b^{v(n)/v(p)}v(z) + \frac{b^{v(n)/v(p)} - 1}{b-1}h & \mbox{if }v(n) \le jv(p) \\
\end{array} \right. .
\end{equation}
Furthermore, 
\begin{enumerate}[label=(\roman{*}), ref=(\roman{*})]

  \item $b$ is the smallest power of $T$ with a coefficient not divisible by $p$ in the series $[p]T$, and $h$ is the valuation of said coefficient.   If no such integer exists, then $b=1$ and $h=0$.  Otherwise, $p \mid b$ and $b >1$.
  \item If $b = 1$, then $j=0$.  If $b \neq 1$, then $j$ is the smallest non-negative integer such that 
\[
v(p) \le b^j\left( (b-1)v(z) + h \right).
\]
\item \label{item:vals-w0} $w = 0$ unless $b>1$ and $v(p) = b^{j}\left( (b-1)v(z) + h \right)$, in which case
\[
w = v\left( \frac{[p^{j+1}]z}{([p^j]z)^p} \right) - h,
\]
which may be equal to $\infty$.
\end{enumerate}

\begin{remark}
  \label{remark:form-gps}
  For the formal additive group, given by $f(X,Y) = X+Y$, the series for multiplication-by-$m$ is $[m]T = mT$, so $b=1$, $h=j=w=0$, and therefore \eqref{eqn:vals-form} simplifies to $v([n]z) = v(z) + v(n)$.

  For the formal multiplicative group, given by $f(X,Y) = (X+1)(Y+1)-1$, multiplication-by-$m$ is 
  \[
  [m]T = (T+1)^m - 1 = T^m + mT^{m-1} + {{m}\choose{2}} T^{m-2} + \cdots + mT,
  \]
  so $b = p$ and $h=0$.
  	
  The formal group of an elliptic curve in standard Weierstrass form is given by 
  \begin{align*}
    f(X,Y) &= X + Y - a_1XY - a_2(X^2Y + XY^2) \\ &+ 2a_3(X^3Y+XY^3) + (a_1a_2 - 3a_3)X^2Y^2 + \ldots.
  \end{align*}
  In particular, it may occur that one or more of the conditions $h>0$, $b > p$ and $j \neq 0$ may hold, for example, over a highly ramified $2$-adic field.  See Examples \ref{example:identity} and \ref{example:sing-pot-good}.
\end{remark}
\end{lemma}

\begin{proof}[Proof of Lemma \ref{lemma:vals-form}]
By \cite[Proposition IV.2.3(a)]{\Sil}, multiplication-by-$n$ has the form
\[
[n]T = nT + O(T^2).
\]
Suppose $n$ is coprime to $p$.  Since $v(z) > 0$, we obtain
\[
v([n]z) = v(z).
\]
Since $[m_1m_2]T = [m_1]([m_2]T)$, it therefore suffices to consider only $n$ equal to a power of $p$.  Let $a_k = v([p^k]z)$ for all non-negative $k$. 

By \cite[Corollary IV.4.4]{\Sil}, the formal group law for $[p]$ has the form
\begin{equation}
  \label{eqn:pt}
[p]T = pf(T) + g(T^p),
\end{equation}
where $f$ and $g$ have no constant term.  We may also assume that the coefficients in $g$ are not divisible by $p$.  By \cite[Proposition IV.2.3(a)]{\Sil},
\[
f(T) = T + O(T^2).
\]
Let $b \in \ZZ$ be the smallest power of $T$ in $g(T^p)$ with a non-zero coefficient, and let $h$ be the valuation of that coefficient (so, in particular, $0 \le h < v(p)$).  Let us momentarily skip the case that $g\equiv0$, so that we have $p \mid b$ and $b \ge p > 1$.  Define $j$ to be the smallest non-negative integer such that
\[
v(p) \le b^j\left( (b-1)v(z) + h \right).
\]
For the moment, let us also assume that the inequality is not an equality.

From \eqref{eqn:pt},
\begin{equation}
\label{eqn:min}
v([p]z) \ge \min \{ v(z) + v(p), bv(z)  + h\}
\end{equation}

Suppose that $j > 0$.  Then, since $v(p) > (b-1)v(z) + h$, the second option determines the minimum in \eqref{eqn:min}, in which the inequality is an equality, and so
\[
a_1 = ba_0 + h
\]

Repeating this argument for all $k \le j$, we find
\[
a_1 = ba_0 + h \implies a_2 = b^2a_0 + bh + h \implies \cdots \implies a_j = b^ja_0 + \frac{b^j-1}{b-1}h.
\]

For $k = j + 1$, we again obtain \eqref{eqn:min} (where we replace $z$ with $[p^{j}]z$), but $v(p) < b^j\left( (b-1)v(z) + h \right)$, so the first option determines the minimum, again where inequality is equality, which implies that
\[
a_{j+1} = a_j + v(p) = b^{j}a_0 + \frac{b^j-1}{b-1}h + v(p).
\]

Repeating this argument, we find that for all $k > j$,
\[
a_k = b^{j}a_0 + \frac{b^j-1}{b-1}h + (k-j)v(p),
\]
from which the result follows with $w = 0$.

Now suppose that $v(p) = b^{j}\left( (b-1)v(z) + h \right)$.  This gives $v(p) = (b-1)a_j + h$. The only place in which this affects the proof is the application of \eqref{eqn:min} for $k=j+1$.  
In this case, the minimum in \eqref{eqn:min} compares two equal values and we obtain instead the alternate form
\[
a_{j+1} = a_{j} + v(p) + w,
\]
for $w$ either $\infty$ (if $[p^{j+1}]z=0$) or a non-negative integer.  
If $w \neq \infty$, then we find that
\begin{align*}
w &= a_{j+1} - a_j - v(p) \\
%  &= a_{j+1} - a_j - (p-1)b^jv(z) \\
  &= a_{j+1} - a_j - (b-1)a_j - h\\
  &= v([p^{j+1}]z) - bv([p^j]z) - h.
\end{align*}
For $k > j+1$,
\[
a_k = a_{k-1} + v(p)
\]
as before.  Combining this with the other cases yields the general formula.

Finally, we return to the case that $g \equiv 0$.  In this case, \eqref{eqn:min} is replaced with
\[
v([p]z) = v(z) + v(p),
\]
and we obtain the formula with $b=1$, $h=0$, $j=0$ and $w=0$.
\end{proof}

The formula in Lemma \ref{lemma:vals-form} being somewhat cumbersome, we set some notation for the class of such sequences.

\begin{definition}
\label{defn:sn}
Suppose $p \in \ZZ$ is a prime, and let $u$ be the valuation on $\QQ$ associated to $p$.  Suppose  
\[
b \in p\ZZ^{>0} \cup \{1\}, \quad d \in \ZZ^{>0}, \quad h \in \ZZ^{\ge0}, \quad s \in \ZZ^{> 0} \cup \{ \infty \}, \quad w \in \ZZ^{\ge 0} \cup \{ \infty \}.
\]
If $b=1$, set $j=0$.  Otherwise, let $j$ to be the smallest non-negative integer such that 
\[
d \le b^j\left( (b-1)s + h \right).
\]
Define a sequence in $\ZZ \cup \{ \infty \}$,
\begin{equation*}
\label{eqn:sn}
S_n(p,b,d,h,s,w) = \left\{
\begin{array}{ll}
  b^js + \frac{b^j-1}{b-1}h + d( u(n) - j) + w & u(n) > j \\
b^{u(n)}s + \frac{b^{u(n)}-1}{b-1}h & u(n) \le j \\
\end{array} \right. .
\end{equation*}
\end{definition}

We record a few properties whose proofs are immediate.

\begin{proposition}
  \label{prop:snprop}
  \begin{enumerate}[label=(\roman{*}), ref=(\roman{*})]
    \item \label{item:simple} If $j=0$, then $S_n(p,b,d,h,s,0) = s + du(n)$.
    \item \label{item:snhomog} For any integer $k$, $S_n(p,b,kd,kh,ks,kw) = kS_n(p,b,d,h,s,w)$.
    \item \label{item:sngrowth} For fixed integers $p$, $b$, $d$, $h$, $s$, and $w$, $S_n(p,b,d,h,s,w) = O(\log n)$.
  \end{enumerate}
\end{proposition}

\section{Non-singular reduction}
\label{sec:good}

The sequence $v(W_n)$ for a point of non-zero non-singular reduction has been described in various contexts in \cite[Theorem 1]{MR1654780} (but see Remark \ref{remark:thm-errors}), \cite[Lemma]{\CheHahTwo}, \cite[Lemma 2.6]{MR2377127}, \cite[Lemma 5]{MR2747036}, \cite[Lemma 3.4]{MR2377368}.  Loosely speaking, in most cases one expects that for non-torsion points with non-singular reduction of order $n_P>1$,
\begin{equation}
  \label{eqn:loosely}
v(W_n) = \left\{ \begin{array}{ll}
v(W_{n_P}) + v(n/n_P) & \mbox{if } n_P \mid n \\
0 & \mbox{if } n_P \nmid n \\
\end{array} \right. .
\end{equation}
There are exceptions, however.  Lemma \ref{lemma:vals-form} on formal groups allows us to prove a somewhat stronger, more general statement.  Please refer to Definition \ref{defn:sn} for the sequence $S_n$, which generalises \eqref{eqn:loosely}.

\begin{theorem}
\label{thm:non-sing}
Assume that $E$ is in minimal Weierstrass form, $P$ has non-singular reduction, and let $n_P$ be the smallest non-negative integer such that $\widetilde{[n_P]P} = \widetilde{\mathcal{O}}$ over the residue field $\mathbf{k}$.  There exist 
\[
b_P \in p\ZZ^{>0} \cup \{ 1 \}, \quad h_P \in \ZZ^{\ge 0}, \quad s_P \in \ZZ^{>0} \cup \{ \infty \}, \quad w_P \in \ZZ^{\ge 0} \cup \{ \infty \},
\]
such that
\begin{equation}
\label{eqn:non-sing}
v(W_n) =  
 \min\left\{ 0, \frac{v(x(P))}{2} \right\}n^2 +
\left\{ \begin{array}{ll}
  S_{n/n_P}(p,b_P,v(p),h_P,s_P,w_P) & \mbox{if } n_P \mid n \\
0 & \mbox{if } n_P \nmid n \\
\end{array}
\right..
\end{equation}
Furthermore, $v(x(P)) < 0$ if and only if $n_P = 1$.
\end{theorem}

\begin{corollary}
\label{cor:torsion}
Assume that $P$ is a non-trivial \emph{torsion} point with non-singular reduction to a point of order $n_P > 1$.  Suppose that $E$ is a minimal Weierstrass model.  Then
\begin{equation*}
\label{eqn:non-sing-tors}
v(W_n) =
\left\{ \begin{array}{ll}
\infty & \mbox{if } n_P \mid n \\
0 & \mbox{if } n_P \nmid n \\
\end{array}
\right. .
\end{equation*}
\end{corollary}

\begin{remark}
  This corollary implies that the non-zero terms of an elliptic divisibility sequence associated to an integral torsion point $P$ are supported only on the primes of bad reduction.  If $P$ is a non-integral torsion point (necessarily of order $2$), then the non-zero terms of the EDS are supported on primes of bad reduction, and $2$.  See Example \ref{example:2-torsion}.
  \label{remark:torsionreduction}
\end{remark}

\begin{corollary}
\label{cor:non-sing-nice}
Assume that $P$ is a \emph{non-torsion} point with non-singular reduction to a point of order $n_P > 1$.  Suppose that $E$ is a minimal Weierstrass model.  Under any of the three below listed conditions,
\begin{equation}
\label{eqn:non-sing-nice}
v(W_n) =
\left\{ \begin{array}{ll}
 v(W_{n_P}) + v(n/n_P) & \mbox{if } n_P \mid n \\
0 & \mbox{if } n_P \nmid n \\
\end{array}
\right. .
\end{equation}
\begin{enumerate}[label=(\roman{*}), ref=(\roman{*})]

\item $v(p) < (p-1)v(W_{n_P})$.
\item $K=\QQ_p$,  and we are not in the special case that $p=2$, $v(W_{n_P})=1$ and $E$ has ordinary or multiplicative reduction.
\item $K$ is unramified over $\QQ_p$ and $p \ge 3$.
\end{enumerate}
\end{corollary}

\begin{remark}
  \label{remark:thm-errors}
  The statement of \cite[Theorem 1]{MR1654780} corresponding to this paper's Theorem \ref{thm:non-sing} is incorrect, in that it holds only under the missing assumption that $E$ is in minimal Weierstrass form.  This assumption is required when applying \cite[Proposition VII.2.2]{\Sil} during the proof of their Lemma 1.  Furthermore, they give the simpler form \eqref{eqn:non-sing-nice} while neglecting to include the assumption that $v(p) < (p-1)v(W_{n_P})$ or an equivalent.  These omissions are corrected in the later paper \cite{\CheHahTwo}.
\end{remark}

We now prove Theorem \ref{thm:non-sing} and its corollaries.

\begin{proof}[Proof of Theorem \ref{thm:non-sing}]

  There is a standard isomorphism of groups (see \cite[Proposition VII.2.2]{\Sil})
\[
\Theta: E_1(K) \rightarrow \hat{E}(\mathcal{M}), \quad (x,y) \mapsto -\frac{x}{y},
\]
where $\hat{E}(\mathcal{M})$ is the formal group of $E$.  Write
      \begin{equation}
	\label{eqn:np}
      [n]P = \left( \frac{\phi_n}{\Psi_n^2}, \frac{\omega_n}{\Psi_n^3} \right)
    \end{equation}
as in Section \ref{sec:divpoly-back}.

We begin with the case that $n_P = 1$, i.e. $\widetilde{P} = \widetilde{\mathcal{O}}$.  Writing $P = (x,y)$, we have $n_P = 1$ if and only if $v(x) < 0$.  From the minimal Weierstrass equation for $E$, we obtain $2v(x) = 3v(y)$.  Let $v_0 = v(y) - v(x) = \frac{1}{2}v(x)$.  Then $v(x) = 2v_0$ and $v(y) = 3v_0$.  Since $v_0 < 0$, by Proposition \ref{prop:divpoly}\ref{item:valsdivpoly}, $v(\phi_n) = 2n^2v_0$, from which we obtain
  \[
  v(\Theta([n]P)) = v(x/y) = -\frac{1}{2}v(x) = v(\Psi_n) - n^2v_0.
  \]
  Then, 
  \begin{equation}
    \label{eqn:np1}
  v(\Psi_n) = \frac{v(x(P))}{2}n^2 + v(\Theta([n]P)). 
\end{equation}

Now suppose that $\widetilde{P} \neq \widetilde{\mathcal{O}}$ instead.  Then we have $v(\Psi_n), v(\phi_n), v(\omega_n) \ge 0$ (by the definition of the division polynomials).  A theorem of Ayad \cite[Theorem A]{\AyadS} implies that since $P$ has non-zero non-singular reduction, at least one of $v(\Psi_n)$ and $v(\phi_n)$ is zero for each $n$.  It follows that $n= n_P$ is the smallest positive $n$ such that $v(\Psi_{n}) > 0$, and that $v(\phi_{kn_P}) = 0$.  This implies $v(\omega_{kn_P}) = 0$ by \eqref{eqn:np} and the form of the Weierstrass equation.  Therefore for all integers $k$,
\begin{equation}
  \label{eqn:np2}
v(\Theta([kn_P]P)) = v(-\phi_{kn_P}\Psi_{kn_P}/\omega_{kn_P}) = v(\Psi_{kn_P}).
\end{equation}

Now, in both cases ($n_P=1$ and $n_P \neq 1$), Lemma \ref{lemma:vals-form}, written in terms of Definition \ref{defn:sn}, says
\[
v(\Theta([kn_P]P)) = S_k(p,b_P,v(p),h_P,v(\Theta([n_P]P)),w_P)
\]
for some $b_P$, $h_P$, and $w_P$.  If $n_P \nmid n$, then $v(\Psi_n) = 0$.  This, combined with \eqref{eqn:np1} and \eqref{eqn:np2}, completes the proof.
\end{proof}

\begin{proof}[Proof of Corollary \ref{cor:torsion}]
  Since $n_P > 1$, we have $v(W_n) = 0$ whenever $n_P \nmid n$.  Furthermore, $s_P = v(W_{n_P}) = \infty$ since $W_{n_P} = 0$ by Proposition \ref{prop:divtorsion}.  
\end{proof}

\begin{proof}[Proof of Corollary \ref{cor:non-sing-nice}]

  In all three parts, we use the notation of Theorem \ref{thm:non-sing} and Lemma \ref{lemma:vals-form}.

  {\bf Condition (1).}  Since $P$ is non-torsion, $v(W_{n_P}) \neq \infty$.  Since $n_P > 1$,  
  \[
v(W_n) = 
\left\{ \begin{array}{ll}
  S_{n/n_P}(p,b_P,v(p),h_P, v(W_{n_P}),w_P) & \mbox{if } n_P \mid n \\
0 & \mbox{if } n_P \nmid n \\
\end{array}
\right. .
\]
Condition (1) implies that $j=0$ (since $h\ge0$) in Definition \ref{defn:sn} (since this is immediate if $b_P=1$, while otherwise $b_P \ge p$ from which it follows).  Condition (1) also implies $w_P=0$ by Lemma \ref{lemma:vals-form}\ref{item:vals-w0}.  By Proposition \ref{prop:snprop}\ref{item:simple}, 
\[
S_k(p,b_P,v(p),h_P,v(W_{n_P}),0) = v(W_{n_P}) + v(p)u(k) = v(W_{n_P}) + v(k).
\]

{\bf Condition (2).}  For $\QQ_p$, $h_P=0$ by definition.  We have $v(p)=1$ and so
\[
(p-1)v(W_{n_P}) \ge p-1 \ge 1 = v(p).
\]
Furthermore, the overall inequality between leftmost and rightmost is strict (and hence we are in condition (1) and we are done) except possibly in the case that $p=2$ and $v(W_{n_P})=1$.  Either way, $j=0$.  Then, according to Lemma \ref{lemma:vals-form} and the proof of Theorem \ref{thm:non-sing}, 
\[
w_P = v([2]z/2z),
\]
where $z = \Theta([n_P]P)$.  If $w_P = 0$, we are done as in condition (1).  For an elliptic curve, the formal group law is
\[
[2]z = 2z - a_1z^2 - 2a_2z^3 + O(z^4)
\]
so that in the case that $a_i \in R$, $p=2$ and $v(z) = 1$, 
\[
v([2]z/2z)>0 \iff v(1-a_1z/2) > 0 \iff v(a_1) = 0.
\]
(Recall that the residue field $\mathbf{k}$ has only one unit since $p=2$.)  As remarked in the proof of \cite[Lemma 5]{MR2747036},
\[
v(a_1) = 0 \iff E\mbox{ has ordinary or multiplicative reduction}.
\]  

{\bf Condition (3).}  Since $p \ge 3$, and $K$ is unramified over $\QQ_p$,
\[
v(p) \le v(W_{n_P}) < (p-1)v(W_{n_P})
\]
and therefore condition (1) is satisfied.
\end{proof}

\section{Singular reduction on a curve of potential good reduction}
\label{sec:pot-good}

In the case that $P$ has singular reduction but $E$ has potential good reduction, we can extend the field and change coordinates to obtain a minimal Weierstrass equation of good reduction for the curve.  Keeping track of the effect this has on the elliptic divisibility sequence allows us to give a formula for the valuations $v(W_n)$.  %In short, there is some positive integer $d$ dividing $24$ such that $dv(W_n)$ is of the form \eqref{eqn:non-sing} given in Theorem \ref{thm:non-sing}.  The statement is made precise as follows.

\begin{theorem}
  Assume that $E$ has potential good reduction, and $P \in E(K)$ has EDS $W_n$.  There exists an isomorphism $\phi: E \rightarrow E'$ to an elliptic curve in minimal Weierstrass form with good reduction, such that $\phi$ is defined over a finite extension $L$ of $K$, with ramification degree $d \mid 24$.  Let $v_1$ be a valuation of $L$ lying over $v$, such that $v_1(z) = dv(z)$ for $z \in K$, and let $W_n'$ be the EDS associated to $\phi(P)$.  Then,
  \begin{equation}
    \label{eqn:pot-good}
  dv(W_n) = (n^2-1)dv(\Delta_E)/12 + v_1(W_n')
\end{equation}
  where $v_1(W_n')$ is of the form \eqref{eqn:non-sing} of Theorem \ref{thm:non-sing}, since $\phi(P)$ has non-singular reduction.  Furthermore, $12 \mid dv(\Delta_E)$. 
  \label{thm:pot-good}
\end{theorem}

\begin{proof}[Proof of Theorem \ref{thm:pot-good}]

By the theory of reduction types, over some extension $L/K$, and under an appropriate change of coordinates defined over $L$, we obtain an isomorphic curve and point $E'$ and $P'$ having good reduction.  Let $v_1$ be the valuation for $L$ lying above $v$ such that $v_1(z) = dv(z)$ for $x \in K$, where $d$ is the degree of ramification of $L$ over $K$.  Then $dv(W_n)$ is the valuation of the sequence associated to $E$ and $P$ considered over $L$.  Changing coordinates, using Theorem \ref{thm:change-to-minimal}, we obtain
  \[
  dv(W_n) = (n^2-1)r + v_1(W_n')
  \]
  for some $r \in \ZZ$ such that $0 = v_1(\Delta_{E'}) = dv(\Delta_E)- 12r$.  If $E$ has bad reduction, the extension $L/K$ is ramified, so that $d > 1$ \cite[Proposition VII.5.4(a)]{\Sil}.  Furthermore, we can choose $L$ so that $d$ divides $12$ by changing to Legendre normal form \cite[Proofs of Propositions III.1.7(a), VII.5.4(c)]{\Sil}, unless $p=2$, in which case $d$ divides $24$ by changing to Deuring normal form \cite[Proofs of Propositions A.1.3 and A.1.4(a)]{\Sil}.  Even if $E/K$ was minimal, $E/L$ will not be minimal.  %The valuation of $\Delta_E$ in $L$ is $d$ times greater; $r$ must be chosen to change this to $0$.  Hence $r = dv(\Delta_E)/12$. 
\end{proof}

\begin{remark}
  Ayad shows that a $P$ of non-trivial reduction on a minimal curve has singular reduction if and only if $v(W_n) > 0$ for all $n \ge 2$ \cite[Theorem A]{\AyadS}.  In the above theorem, if $P$ has singular reduction, we do indeed obtain sequences satisfying $v(W_n) > 0$ for all $n \ge 2$.  See Example \ref{example:sing-pot-good}.
\end{remark}

The following proposition gives some restrictions on the parameters used in Theorem \ref{thm:pot-good}.

\begin{proposition}
  \label{prop:extras-pot-good}
  Under the hypotheses and notations of Theorem \ref{thm:pot-good}, let %with reference to the parameters of \eqref{eqn:non-sing}, let 
  \[
  %d' = \frac{d}{\gcd(d,dv(\Delta_E)/12)}.
  d' = 12 / \gcd(v(\Delta_E),12).
  \]
  Then,
\begin{enumerate}[label=(\roman{*}), ref=(\roman{*})]

  \item If $d' =2, 4$, then $n_P \in \{ 1, 2 \}$.
  \item If $d' =3$, then $n_P \in \{ 1, 3 \}$.
  \item If $d' = 6, 12$, then $n_P = 1$.
  \end{enumerate}
\end{proposition}

For Proposition \ref{prop:extras-pot-good}, we require an elementary number theoretical lemma.

\begin{lemma}
\label{lemma:n-squared}
Let $a, b \in \ZZ^{>0}$.  Suppose that for all integers $n$,
\begin{equation}
\label{eqn:n-squared}
n \not\equiv 0 \pmod a \implies n^2 \equiv 1 \pmod b.
\end{equation}
Then
\[
(a,b) \in \{ (1,*), (*,1), (2,2), (3,3), (2,4), (2,8) \}
\]
where $*$ represents any positive integer.
\end{lemma}

\begin{proof}
        The statement \eqref{eqn:n-squared} holds vacuously if $a=1$ and trivially if $b=1$.   If $a=2$, choosing $n=3$ implies $b \in \{1, 2,4,8\}$.  If $a=3$, choosing $n=2$ implies $b \in  \{ 1, 3\}$.  If $a > 3$, choosing $n=2$ and $n=3$ implies that $b=1$. 
\end{proof}

\begin{proof}[{Proof of Proposition \ref{prop:extras-pot-good}}]
        We may assume that $n_P > 1$.  By Theorem \ref{thm:pot-good} and \ref{thm:non-sing}, there are parameters $d, b, h, s, w$ such that
   \begin{equation}
           \label{eqn:pot-good-again}
  dv(W_n) = (n^2-1)dv(\Delta_E)/12 + \left\{ \begin{array}{ll}
                        S_{n/n_P}(p, b, dv(p), h, s, w) & \mbox{ if $n_P \mid n$} \\
                                                              0& \mbox{ if $n_P \nmid n$}
                \end{array} \right. .
\end{equation}
Write $r = dv(\Delta_E)/12 \in \ZZ$, and $g = \gcd(d,r)$.  It is an exercise in elementary number theory to see that $d/g = \gcd(v(\Delta_E),12)$.  The exercise is as follows:  if $y \mid dv$, then $d / \gcd(d, dv/y) = y/\gcd(v,y)$.

The first claim is that without loss of generality, we may assume \eqref{eqn:pot-good-again} holds for some (possibly different) $d,b,h,s,w$ having $g=1$.  Since $v(W_n) \in \ZZ$, we find that for $n$ divisible by $n_P$,  
\[
S_{n/n_P}(p,b,dv(p),h,s, w) 
%\equiv - (n^2-1)r 
\equiv 0 \pmod{g}.
\]
If $n=n_P$, we find that $g \mid s$, and from this we deduce that $g \mid h$ (if $j=0$ we may simply change $h$ without changing the function $S_{n/n_P}$; otherwise take $n$ satisfying $v(n/n_P) = 1$).  Then $g \mid w$ by taking $n$ having $v(n/n_P)$ large enough.  Therefore, by Proposition \ref{prop:snprop}\ref{item:snhomog},
\[
S_{n/n_P}(p,b,dv(p),h,s,w) = g S_{n/n_P}(p,b,dv(p)/g,h/g,s/g,w/g).
\]
Therefore we can divide \eqref{eqn:pot-good-again} by $g$, i.e. replace $d$ by $d/g = \gcd( 12, v(\Delta_E))$, and parameters $h,s,w$ replaced by their own quotients with $g$.  This is the proof of the claim.  Therefore, without loss of generality let us assume $\gcd(d,r)=1$ and $d = \gcd(12, v(\Delta_E))$.

Since $v(W_n)$ is an integer and $r=dv(\Delta_E)/12$ is an integer, 
\[
n \not\equiv 0 \pmod{n_P} \implies (n^2 - 1)r \equiv 0 \pmod{d}.
\]
Since $r$ is coprime to $d$, Lemma \ref{lemma:n-squared} implies that the implication only holds if
\[
(n_P, d) \in \{(1,*),(*,1),(2,2),(2,4),(2,8),(3,3)\}.
\]
\end{proof}

\section{Elliptic troublemaker sequences}
\label{sec:ets}

In this section we define a class of integer sequences which will be needed to describe $v(W_n)$ for points $P$ of singular reduction on an elliptic curve $E$ with multiplicative or potential multiplicative reduction.

\begin{definition}
\label{def:torsionpower}
To any pair $(a, \ell)$ of integers satisfying $\ell \neq 0$, we associate an integer sequence called the \emph{elliptic troublemaker sequence}, defined for $n \ge 0$ by
\begin{equation}
  \label{eqn:rndef}
R_n(a,\ell) 
=
\left\lfloor
\frac{n^2 \widehat{a} (\ell -\widehat{a})}{2\ell}
\right\rfloor 
- \left\lfloor \frac{\widehat{na}(\ell-\widehat{na})}{2\ell} \right\rfloor ,
\end{equation}
where $\widehat{x}$ denotes the least non-negative residue of $x$ modulo $\ell$.

\end{definition}

Some examples are given in Table \ref{rn}.  We devote the rest of this section to properties of these sequences.

\begin{table}
\caption{Elliptic troublemaker sequences $R_n(a,\ell)$ for various $(a,\ell)$.}
\label{rn}
\begin{tabular}{r|rrrrrrrrrrrrr}
$n$:  & 1 & 2 & 3 & 4 & 5 & 6 & 7 & 8 & 9 & 10 & 11 & 12 & 13\\
\hline
$R_n(1,2)$: & 0 & 1 & 2 & 4 & 6 & 9 & 12 & 16 &
20 & 25 & 30 & 36 & 42 \\
$R_n(1,3)$: & 0 & 1 & 3 & 5 & 8 & 12 & 16 & 21
& 27 & 33 & 40 & 48 & 56 \\
$R_n(2,3)$: & 0 & 1 & 3 & 5 & 8 & 12 & 16 & 21
& 27 & 33 & 40 & 48 & 56 \\
$R_n(1,4)$: & 0 & 1 & 3 & 6 & 9 & 13 & 18 & 24
& 30 & 37 & 45 & 54 & 63 \\
$R_n(2,4)$: & 0 & 2 & 4 & 8 & 12 & 18 & 24 & 32
& 40 & 50 & 60 & 72 & 84 \\
%$R_n(3,4)$: & 0 & 1 & 3 & 6 & 9 & 13 & 18 & 24
%& 30 & 37 & 45 & 54 & 63 \\
$R_n(1,5)$: & 0 & 1 & 3 & 6 & 10 & 14 & 19 & 25
& 32 & 40 & 48 & 57 & 67 \\
$R_n(2,5)$: & 0 & 2 & 5 & 9 & 15 & 21 & 29 & 38
& 48 & 60 & 72 & 86 & 101 \\
$R_n(1,6)$: & 0 & 1 & 3 & 6 & 10 & 15 & 20 & 26
& 33 & 41 & 50 & 60 & 70 \\
$R_n(2,6)$: & 0 & 2 & 6 & 10 & 16 & 24 & 32 & 42
& 54 & 66 & 80 & 96 & 112 \\
$R_n(3,6)$: & 0 & 3 & 6 & 12 & 18 & 27 & 36 & 48
& 60 & 75 & 90 & 108 & 126 \\
$R_n(1,7)$: & 0 & 1 & 3 & 6 & 10 & 15 & 21 & 27
& 34 & 42 & 51 & 61 & 72 \\
$R_n(2,7)$: & 0 & 2 & 6 & 11 & 17 & 25 & 35 & 45
& 57 & 71 & 86 & 102 & 120 \\
$R_n(3,7)$: & 0 & 3 & 7 & 13 & 21 & 30 & 42 & 54
& 69 & 85 & 103 & 123 & 144 \\
$R_n(1,11)$: & 0 & 1 & 3 & 6 & 10 & 15 & 21 & 28
& 36 & 45 & 55 & 65 & 76 \\
\end{tabular}
\end{table}

\begin{proposition}
\label{prop:rn}
The function $R_n(a,\ell)$ has the following properties.
\begin{enumerate}[label=(\roman{*}), ref=(\roman{*})]
  \item \label{item:rn0} $R_n(0,\ell) = 0$.
\item \label{item:r01} $R_0(a,\ell) = R_1(a,\ell) = 0$. 
  \item \label{item:aell} $R_n(a,\ell) = R_n(\ell \pm a,\ell)$.
  \item \label{item:ka} For any positive integer $k$, $R_n(ka, k\ell) = kR_n(a,\ell)$.
\item \label{item:quad} For positive integers $n$ and $m$,
$R_n(ma,\ell) = R_{nm}(a,\ell) - n^2R_m(a,\ell)$.
\item \label{item:quadbound} $R_{n+1}(a,\ell)+R_{n-1}(a,\ell)-2R_n(a,\ell) < \ell$.
\item \label{item:triang} For $0 < n < \ell/a$, 
  \[
  R_n(a,\ell) = \frac{n^2-n}{2}a.
  \]
  \item \label{item:floor} 
If $\ell \mid na$ or if $0 \le a < \ell \le 7$, then 
\begin{equation}
  \label{eqn:floorexact}
  R_n(a,\ell)
= \left\lfloor
\frac{n^2 a (\ell -a)}{2\ell}
\right\rfloor .
\end{equation}

\item \label{item:mainalt}  An alternative formula for $R_n(a,\ell)$ is
\begin{multline}
\label{eqn:mainalt}
R_n(a,\ell) = \frac{\ell}{2} \left(
\left( \frac{na}{\ell} - \nal \right)^2
- \left( \frac{na}{\ell} - \nal \right) \right.\\
\left. - n^2 \left( \frac{a}{\ell} - \al \right)^2
+ n^2 \left( \frac{a}{\ell} - \al \right) \right) .
\end{multline}

  \item \label{item:shortmain} If $0 \le a < \ell$, then an alternate formula for $R_n(a,\ell)$ is
\begin{equation}
  \label{eqn:shortmain}
R_n(a,\ell) = \frac{\ell}{2} \left(
\left( \frac{na}{\ell} - \nal \right)^2
- \left( \frac{na}{\ell} - \nal \right) \right.
\left. + \frac{n^2a(\ell-a)}{\ell^2} \right) .
\end{equation}

\item \label{item:bern} 
  Let $B_2(t) = t^2 - t + \frac{1}{6}$, called the \emph{second Bernoulli polynomial}, and let $ \widetilde{B}_2(t)=B_2(t- \lfloor t \rfloor)$, called the \emph{periodic second Bernoulli polynomial}.  Then an alternate formula for $R_n(a,\ell)$ is
\begin{equation*}
\label{eqn:perber}
R_n(a,\ell) = \frac{\ell}{2} \left( \widetilde{B}_2\left( \frac{na}{\ell} \right) - n^2 \widetilde{B}_2 \left( \frac{a}{\ell} \right) + \frac{n^2 - 1}{6} \right).
\end{equation*}

\item \label{item:sumform} An alternate formula for $R_n(a,\ell)$ is
\begin{equation}
\label{eqn:rsums}
R_n(a,\ell)
= \frac{n^2-n}{2} a
+ \left( \sum_{k=1}^{\nal} k\ell  - na \right)
- n^2\left( \sum_{k=1}^{\al} k\ell  - a \right).
\end{equation}
\item \label{item:growth} We have
  \[
  \left| R_n(a,\ell) -  \left( \frac{ \widehat{a} ( \ell - \widehat{a} )}{2\ell} \right) n^2 \right| \le \frac{\ell}{8},
  \]
  where $\widehat{x}$ denotes the least non-negative residue of $x$ modulo $\ell$.
\end{enumerate}
\end{proposition}

\begin{proof}

  We prove the various parts out of order according to the various interdependencies.

  {\bf Parts \ref{item:rn0}, \ref{item:r01}, \ref{item:aell} and \ref{item:floor}}  Direct calculations from the definition.
  
  {\bf Part \ref{item:quadbound}}  For $0 \le x \le 1$, $0 \le x(1-x)/2 \le 1/8$, so that as $b$ ranges through the least non-negative residues modulo $\ell$,
  \begin{equation}
    \label{eqn:floorbd}
  0 \le \frac{ b(\ell-b) }{2\ell} \le \ell/8.
\end{equation}
  For any $A$ and $B$, 
  \begin{equation}
    \label{eqn:floorprops}
  \lfloor A \rfloor + \lfloor B \rfloor \le \lfloor A+B \rfloor, \quad \lfloor A \rfloor + \lfloor -A \rfloor = -1.
\end{equation}
From the definition, \eqref{eqn:floorbd} and \eqref{eqn:floorprops},
\begin{equation*}
  R_{n+1}(a,\ell) + R_{n-1}(a,\ell) - 2R_n(a,\ell)
  \le \left\lfloor \frac{a (\ell-a)}{\ell} \right\rfloor - 2 + 2 \left( \frac{\ell}{8}\right) 
  < \ell.
\end{equation*}

{\bf Part \ref{item:shortmain} }  We will show that \eqref{eqn:shortmain} is equal to \eqref{eqn:rndef}, under the assumption that $0 \le a < \ell$.  

We consider the case $\ell \mid na$ separately.  In this case, using \eqref{eqn:shortmain}, to show \eqref{eqn:rndef} (or actually \eqref{eqn:floorexact}), we need only check that $\frac{n^2a(\ell -a)}{\ell}$, which is an integer divisible by $\ell$, is even.  Both cases, $2 \mid \ell$ and $2 \nmid \ell$, are immediate.  Therefore we assume that $ \ell \nmid na$.  
  
  We express certain quantities in terms of their integer and fractional parts:  write
  \[
  \frac{na}{\ell} = X + x, \quad \frac{n(\ell-a)}{\ell} = Y + y, \quad \frac{n^2a(\ell-a)}{2\ell} = Z + z
  \]
  where $X, Y, Z$ are integers and $0 < x, y < 1$ and $0 \le z < 1$.  We also know that $x+y = 1$.  Furthermore, $x$ and $y$ are rationals with denominator dividing $\ell$.  We have
  \[
  Z+z = \frac{\ell}{2}(X+x)(Y+y).
  \]
  Write
  \[
  \frac{\ell}{2}(X+x)(Y+y) = \frac{1}{2}(\ell XY + \ell xY + X\ell y) + \frac{\ell}{2}xy.
  \]
  We wish to show that the first of the two terms on the right is an integer.  That is, we want to show the integer $\ell XY + (\ell x)Y + X(\ell y)$ is even (note that $x$ and $y$ are not integers, but $\ell x$ and $\ell y$ are).  We do this by cases.  If $X \equiv Y \equiv 0 \pmod 2$, then the integer is even.  If $X \equiv Y \equiv 1 \pmod 2$, then
  \[
   \ell XY + (\ell x) Y + X (\ell y) \equiv \ell + \ell x + \ell y \equiv 2 \ell \equiv 0.
   \]
   If $X \not\equiv Y$, by symmetry we may assume that $X \equiv 0$ and $Y \equiv 1$.  Since $X + Y = n-1$, we discover that $n \equiv 0$.  Then, since $na = X\ell + x\ell$,
   \[
   \ell XY + (\ell x) Y + X (\ell y) 
   \equiv naY \equiv 0.
   \]
   Thus we have discovered that $\frac{1}{2}(\ell XY + \ell xY + X\ell y)$ is an integer.  
   
   Hence 
   \[
   Z = \frac{1}{2}(\ell XY + \ell xY +X \ell y) + \left\lfloor \frac{\ell}{2}xy \right\rfloor, \quad z = \frac{\ell}{2}xy - \left\lfloor \frac{\ell}{2}xy \right\rfloor.
   \]
   Write $x = s/\ell$ for some $0 < s < \ell$ (in other words, $s = \widehat{na}$).  Noting that $xy = x - x^2$, and substituting for the meaning of $x$, $y$ and $z$ in the second equation, we obtain
   \begin{multline*}
\left(
   \left\lfloor \frac{n^2 a (\ell -a)}{2\ell} \right\rfloor
   - \frac{n^2a(\ell-a)}{2\ell} \right)
\\
-
   \frac{\ell}{2} \left(
   \left( \frac{na}{\ell} - \left\lfloor \frac{na}{\ell} \right\rfloor \right)^2 - 
   \left( \frac{na}{\ell} - \left\lfloor \frac{na}{\ell} \right\rfloor \right)
   \right)
 = 
\left\lfloor
\frac{ s(\ell-s)}{2\ell} 
\right\rfloor ,
\end{multline*}
  which is what was required to prove.

  {\bf Part \ref{item:mainalt}}  If $0 \le a < \ell$, it is immediate that \eqref{eqn:mainalt} reduces to \eqref{eqn:shortmain}.  Therefore, using parts \ref{item:aell} and \ref{item:shortmain}, it suffices to check that \eqref{eqn:mainalt} is independent of the choice of residue of $a$ modulo $\ell$.  But this is a direct calculation (compare the formula \eqref{eqn:mainalt} for $R_n(a,\ell)$ and $R_n(a+\ell,\ell)$).

  {\bf Parts \ref{item:ka} and \ref{item:bern} } Direct calculations from \eqref{eqn:mainalt} of part \ref{item:mainalt}.

{\bf Part \ref{item:quad}}  Letting
\[
S(n) = \frac{\ell}{2}\left(
\left( \frac{na}{\ell} - \nal \right)^2
- \left( \frac{na}{\ell} - \nal \right) \right)
\]
we obtain, from the formula \eqref{eqn:mainalt} of part \ref{item:mainalt},
\begin{align*}
R_n(ma,\ell) &= S(mn) - n^2S(m), \\
 R_{mn}(a,\ell) &= S(mn) - m^2n^2S(1), \\
 n^2R_m(a,\ell) &= n^2S(m) - n^2m^2S(1).
\end{align*}

{\bf Part \ref{item:sumform}} Let
\[
T(n) = \sum_{k=1}^{\nal} k \ell - na.
\]
A calculation reveals that 
\[
T(n)
= \left(\left\lfloor \frac{na}{\ell} \right\rfloor^2 + \left\lfloor \frac{na}{\ell} \right\rfloor - \frac{2na}{\ell} \left\lfloor \frac{na}{\ell} \right\rfloor \right) \frac{\ell}{2} .
\]
Then, expanding formula \eqref{eqn:mainalt} of part \ref{item:mainalt}, we obtain
\[
 R_n(a,\ell) = T(n) - n^2T(1) + \frac{\ell}{2}\left(  - \frac{na}{\ell} + \frac{n^2a}{\ell} \right).
\]

  {\bf Part \ref{item:triang}}  Follows immediately from part \ref{item:sumform}.

{\bf Part \ref{item:growth}}  Immediate from parts \ref{item:bern} and \ref{item:aell}, together with the observation that $X(1-X)$ has maximum $1/4$ on the interval $[0,1]$.
\end{proof}

\begin{remark}
  By Proposition \ref{prop:rn}, parts \ref{item:aell} and \ref{item:ka}, it suffices to study sequences satisfying $0 \le 2a \le \ell$ with $\gcd(a,\ell)=1$.  We could index the collection of such sequences by $\QQ \cap [0,\frac{1}{2}]$.
\end{remark}

\section{Multiplicative reduction}
\label{sec:mult}

We now turn to $P$ having singular reduction on a curve $E$ of multiplicative bad reduction, where we will require the theory of the Tate curve.  

Suppose that $E$ does not have potential good reduction, i.e. $v(j_E) < 0$.  In this case, there is a unique $q \in K^*$ with $v(q) > 0$ such that the Tate curve $E_q$ is isomorphic to $E$ over a finite extension $L$ of $K$.  The case of multiplicative reduction is the case that $L$ can be taken to be unramified, and split multiplicative reduction corresponds to $L=K$.  See \cite[Chapter V]{\Siltwo} for background.

\begin{definition}
  \label{defn:aplp}
  For any elliptic curve $E/K$ with non-integral $j$-invariant, and any $P \in E(K)$, let $\phi: E \rightarrow E_q$ be an isomorphism to the Tate curve (note that $v(q) > 0$).  Analytically, $E_q$ is isomorphic to $K^*/q^\ZZ$, under which $\phi(P)$ corresponds to some $u \in K^*$.  As a convention, choose $u$ satisfying $0 \le v(u) < v(q)$.  Define 
\[
\ell_P = v(q), \quad a_P = v(u).
\]
Note that, despite the notation, the quantity $\ell_P$ only depends on the elliptic curve $E$.  
It is a standard fact that $\ell_P = v(\Delta(E_q)) = -v(j(E_q)) = - v(j(E))$.  

\end{definition}

\begin{remark}
  \label{remark:mult-component}
  Using the standard isomorphisms (see \cite[Remark IV.9.6]{\Siltwo}),
  \[
  E(K)/E_0(K) \rightarrow K^*/q^\ZZ R^* \rightarrow \ZZ/\ell_P\ZZ,
  \]
$a_P = v(u)$ is the component of the N\'eron model special fibre ($\cong \ZZ/\ell_P\ZZ$) containing $P$.  In particular, $a_P = 0$ if and only if $P$ has non-singular reduction.
\end{remark}

\begin{theorem}
\label{thm:mult}
Suppose that $P$ has singular reduction, and $E$ is in minimal Weierstrass form with multiplicative reduction.  Let $a_P$ and $\ell_P$ be as in Definition \ref{defn:aplp}.  Let $n_P$ be the smallest positive integer such that $\widetilde{[n_P]P} = \widetilde{\mathcal{O}}$ over the residue field $\mathbf{k}$.  Then
there exist 
\[
 s_P \in \ZZ^{>0} \cup \{ \infty \}, \quad w_P \in \ZZ^{\ge 0} \cup \{ \infty \}
\]
such that for all positive integers $n$,
\begin{equation}
v(W_n) = 
R_n(a_P,\ell_P)  
+ \left\{ \begin{array}{ll}
  S_{n/n_P}(p,p, v(p), 0, s_P,w_P) & n_P \mid n \\
0 & n_P \nmid n \\
\end{array} \right. 
.
\label{eqn:mult}
\end{equation}
Furthermore,
\begin{enumerate}[label=(\roman{*}), ref=(\roman{*})]
\item \label{item:mult-ns}Letting $g = \gcd(a_P,\ell_P)$, the integers $n_P$ and $s_P$ are given by 
\[
n_P = \frac{\ell_P\ord(q^{a_P/g}u^{-\ell_P/g})}{g}, \quad
s_P = v(1-q^{n_Pa_P/\ell_P}u^{-n_P}). 
\]
where $\ord$ denotes the multiplicative order of the image in the residue field $\mathbf{k}$.
\item \label{item:mult-nplp}
  If $P$ is a torsion point of order $N$, then 
  \[
  n_P = \frac{\ell_P}{\gcd(a_P,\ell_P)} = N.
  \]
\end{enumerate}
\end{theorem}

\begin{proof}
We drop the subscripts and write $\ell = \ell_P$ and $a = a_P$.  
First, we consider the Tate curve $E_q$, by which we mean,  
as described in \cite[Theorem V.3.1]{\Siltwo}, the curve given by the model
\begin{equation}
  E_q: y^2 + xy = x^3 + a_4(q)x + a_6(q), 
  \label{eqn:tatecurve}
\end{equation}
in which case, the point $u$ corresponds to $(X(u,q),Y(u,q))$ where $X$ and $Y$ are defined in \cite[Theorem V.3.1]{\Siltwo}.  We can define $\Psi_n(u,q)$ as the usual polynomial in $X$ and $Y$ for \eqref{eqn:tatecurve}.  
As in \cite[Proposition V.3.2]{\Siltwo}, define the normalised theta function as
\begin{equation*}
\label{eqn:theta}
\theta(u,q) 
= (1-u)\prod_{k \geq 1} \frac{(1-q^ku)(1-q^ku^{-1})}{(1-q^k)^2}.
\end{equation*}
We wish to express $\Psi_n(u,q)$ in terms of the normalised theta function.  Over $\CC$, we have
\begin{equation*}
  \Psi_n(u,q) = (-2\pi i)^{1-n^2}\frac{\sigma(u^n,q)}{\sigma(u,q)^{n^2}},
  \label{eqn:psi}
\end{equation*}
where
\begin{equation*}
  \sigma(u,q) = - \frac{1}{2\pi i}e^{\frac{1}{2}\eta(1)z^2}e^{-\pi i z}\theta(u,q).
  \label{eqn:sigma}
\end{equation*}
(To see this, use \eqref{eqn:cpsi}, together with the standard change of coordinates to eliminate $2 \pi i$, for example in \cite[Section V.1]{\Siltwo}.)  Therefore, over $\CC$,
\begin{equation}
\label{eqn:wn}
\Psi_n(u,q) 
=  u^{(n^2-n)/2} \frac{\theta(u^n,q)}{\theta(u,q)^{n^2}}.
\end{equation}
Using the same method as the proof of \cite[Proposition V.3.2(b)]{\Siltwo}, this relation also holds over $K$ (in fact \cite[Proposition V.3.2(b)]{\Siltwo} is a special case).

We let $W_n = \Psi_n(u,q)$ (so it is the EDS associated to \eqref{eqn:tatecurve} and the point $(X(u,q),Y(u,q))$).  We wish to discover the form of $v(W_n)$.  Note that \eqref{eqn:tatecurve} is always a minimal Weierstrass model in its isomorphism class (this can be verified using \cite[Remark VII.1.1]{\Sil} and the $q$-expansions of $\Delta$ and $c_4$).  Therefore, information about $v(W_n)$ for a Tate curve applies to any EDS associated to an elliptic curve $E/K$ in minimal Weierstrass form and having split multiplicative reduction.

For any $k$ satisfying $na + \ell k  \neq  0 $, we have $ v(1-q^ku^n) = \min\{ \ell k + na, 0 \}$.  For $k > 0$, we have $v(1-q^k) = 0$.  Therefore (recalling that $a,n, \ell \ge 0$), 
\begin{equation*}
\label{eqn:sumfloor}
v(\theta(u^n,q)) 
= t_P(n) + \left( \sum_{k=1}^{\nal} k\ell - na \right),
\end{equation*}
where $t_P(n) = v(1-q^{k}u^{-n})$ for the unique integer $k$ for which $-na + \ell k= 0$, if such an integer exists, and $t_P(n)=0$ otherwise.  Taken together with \eqref{eqn:wn} and \eqref{eqn:rsums}, this gives
\[
v(W_n) = R_n(a,\ell)  + t_P(n) - n^2 t_P(1).
\]
However, $t_P(1) = 0$ since we assumed $0 \le a < \ell$, and there is no integer $k$ with $k\ell = a$.

We will now find an expression for $t_P(n)$.  The equation $k\ell = na$ has no solution in $k$ unless
\[
n_0 := \frac{\ell}{\gcd(a,\ell)}
\]
divides $n$.  Therefore, if $n_0 \nmid n$, then $t_P(n) = 0$.  Even if $n_0 \mid n$, and if we let $k_0$ be such that $k_0 \ell = n_0a$, as long as $q^{k_0}u^{-n_0} \not\equiv 1$ in $\mathbf{k}$, we still have $t_P(n_0) = 0$.  Therefore, let $n_P = b n_0$ where $b$ is the order of $q^{k_0}u^{-n_0}$ in $\mathbf{k}$.  In other words, $n_P$ is the smallest integer such that $t_P(n_P) \neq 0$, and furthermore, if $t_P(n) \neq 0$, then $n_P \mid n$.  This gives the expression for $n_P$ in item \ref{item:mult-ns}.

The statement of the theorem requires that we also show that $n_P$ is the smallest positive integer such that $\widetilde{[n_P]P} = \widetilde{\mathcal{O}}$ in $E(k)$.  But it is a property of division polynomials that $\widetilde{[n]P} = \widetilde{\mathcal{O}}$ exactly when $W_n \equiv 0$ in $\mathbf{k}$, i.e. when $1-q^{k(n)}u^{-n}$ vanishes in $\mathbf{k}$.  Therefore this follows from the previous paragraph.

We return to finding an expression for $t_P(n)$.  At this point, we are reduced to finding an expression for 
\[
t_P'(s) = t_P(sn_P)
\]
Let $k_P$ be the unique integer such that $k_P \ell = n_P a$, and set $\beta = q^{k_P}u^{-n_P}$.  Then $t_P'(s) = v(1 - \beta^s)$.

Let $s_P = v(1-\beta)$, which is positive by construction.  Let $U(K)$ be the kernel of the reduction map $K^* \rightarrow \mathbf{k}^*$.  Let $\mathbb{G}_m$ be the formal multiplicative group.  By the theory of formal groups, we have an isomorphism
\[
U(K) \longrightarrow \mathbb{G}_m(\mathcal{M}),  \quad u \mapsto 1-u.
\]
Restricting the isomorphism
\[
 K^*/q^\ZZ \longrightarrow E_q(K), \quad u \mapsto (X(u), Y(u))
 \]
 to $U(K)$, and recalling the isomorphism
 \[
 \Theta: E_{q,1}(K) \longrightarrow \hat{E_q}(\mathcal{M}), \quad (x,y) \mapsto -x/y,
 \]
 as in the proof of Theorem \ref{thm:non-sing}, we obtain a chain of isomorphisms
 \[
  \xymatrix@R=10pt{ 
\mathbb{G}_m(\mathcal{M}) \ar[r]&  U(K) \ar[r] & E_1(K) \ar[r] & \hat{E}(\mathcal{M}), \\ 
1-u \ar@{|->}[r] & u \ar@{|->}[r] & (X(u),Y(u)) \ar@{|->}[r] & - X(u)/Y(u).
 }
\]
It can be verified by the definitions of $X(u)$ and $Y(u)$ that this chain has the property that 
\[
v(1-u) = v(X(u)/Y(u)).
\]
So $s_P = v(1-\beta) = v(X(\beta)/Y(\beta)) = v(\Theta([n_P]P))$.  (Note that $v(1-u) = v(1-u^{-1})$.)

Lemma \ref{lemma:vals-form} for $\mathbb{G}_m(\mathcal{M})$ implies
\[
t_P'(s) = v(1 - \beta^s) = S_s(p,p, v(p),0, s_P,w_P),
\]
for some $w_P \in \ZZ^{\ge 0} \cup \{ \infty \}$.  (Note that $b=p$ and $h=0$ in Lemma \ref{lemma:vals-form} by Remark \ref{remark:form-gps}.)
Thus, we have shown \eqref{eqn:mult}, for any $E/K$ having a minimal Weierstrass equation and split multiplicative reduction.  We have also found an expression for $s_P$ and $n_P$ (item \ref{item:mult-ns}) in this case.

Now suppose that $E$ has non-split multiplicative reduction.  Then we can let $L/K$ be an unramified extension over which $E$ is isomorphic to $E_q$.
Because the extension is unramified, $E$ is minimal over $L$ because it is minimal over $K$.  Therefore $E$, considered over $L$, satisfies the assumptions of the previous case of split multiplicative reduction.  Letting $v_1$ be the valuation of $L$ lying over $K$ such that $v_1(z) = v(z)$ for $z \in K$, we find that $v(W_n) = v_1(W_n)$ has the form \eqref{eqn:mult}.

In the case that $W_n$ is associated to an $N$-torsion point, $N > 1$, then $u$ must satisfy
$u^N = q^s$
for some integer $s$ coprime to $N$, which implies that $Na_P = s\ell_P$ by considering valuations.  Combined with item \ref{item:mult-ns}, this shows item \ref{item:mult-nplp}.

\end{proof}

\section{Singular reduction on a curve with additive potential multiplicative reduction}
\label{sec:pot-mult}

This section covers the last remaining case, after which the accumulated theorems of Sections \ref{sec:eds-back}, \ref{sec:good}, \ref{sec:pot-good}, \ref{sec:mult} and \ref{sec:pot-mult} combine to give Theorem \ref{thm:main-intro}.  Here, $c_4:= b_2^2 - 24b_4 = (a_1^2+4a_2)^2 - 24(2a_4 + a_1a_3)$ is the usual quantity defined in terms of the coefficients of the Weierstrass equation.

\begin{theorem}
\label{thm:pot-mult}
  Assume that $E$ does not have potential good reduction.  There exists an isomorphism $\phi: E \rightarrow E'$ to an elliptic curve in minimal Weierstrass form with multiplicative reduction, such that $\phi$ is defined over a finite extension $L$ of $K$, with ramification degree $d \mid 24$.  Let $v_1$ be a valuation of $L$ lying over $v$, such that $v_1(z) = dv(z)$ for $z \in K$, and let $W_n'$ be the EDS associated to $\phi(P)$.  Then,
  \begin{equation*}
    \label{eqn:pot-mult}
  dv(W_n) = (n^2-1)dv(c_4(E))/4 + v_1(W_n')
\end{equation*}
where $v_1(W_n')$ is of the form \eqref{eqn:non-sing} of Theorem \ref{thm:non-sing}, if $\phi(P)$ has non-singular reduction, or the form \eqref{eqn:mult} of Theorem \ref{thm:mult}, if $\phi(P)$ has singular reduction.  
\end{theorem}
\begin{proof}
        Suppose that $E$ has additive reduction.  There is some finite extension of $K$ over which $E$ has split multiplicative reduction.  So we have, by Theorem \ref{thm:change-to-minimal},
\[
dv(W_n) = (n^2-1)r + v_1(W_n')
\]
for some $r$.  The extension $L/K$ must be ramified \cite[Proposition VII.5.4]{\Sil}, so $d > 1$.  The extension needed to obtain (not necessarily split) multiplicative reduction has ramification degree dividing $24$ \cite[Proofs of Propositions III.1.7, VII.5.4(c), A.1.3, A.1.4(a)]{\Sil}.  To obtain split reduction may require a further unramified field extension.  Therefore $d \mid 24$.  Although $E/K$ may be minimal, $E/L$ is no longer.  The change of variables required to make it minimal must take $E$ to $E'$ having $v_1(c_4(E')) = 0$.  Therefore, $r = dv(c_4(E))/4$.
\end{proof}

\section{Integral points}
\label{sec:integral}

We begin with some preliminaries about heights.  Let $h(p/q) = \log \max\{ |p|, |q| \}$ be the usual logarithmic height on $\QQ$.  The naive height of a point $P \in E(\QQ)$ is $h(P) = h(x(P))$.  The canonical height of $P$ is
\[
\widehat{h}(P) = \frac{1}{2} \lim_{n \rightarrow \infty} \frac{h(x([2^n]P))}{4^n}.
\]
This definition follows \cite[\S VIII.9]{\Sil} and differs by a factor of $2$ from the definitions of some other authors.  

Lang, \cite[Theorem 2.1]{MR717593}, shows that
  \[
  \left| \frac{h(P)}{2} - \widehat{h}(P) \right|
  \le
  \frac{1}{6}h(E) + O(1).
  \]

The meaning of the notation $h(E)$, the height of $E$, in the introduction is
  \[
h(E) = h_0(E) = \max\{ h(j), \log|\Delta|, 1 \}.
  \]
However, in this section, following Ingram, it is convenient to consider
elliptic curves in short Weierstrass form,
\[
y^2 = x^3 + Ax + B,
\]
with integral coefficients, and to use
  \[
  h(E) = h_I(E) = \max\{ h(j), \log \max \{ 4|A|, 4|B| \} \},
  \]
which is always at least $2 \log 2$.  
  The statement of Theorem \ref{thm:patrick-new} is the same no matter which height is used, thanks to the following proposition, which says the two heights are equivalent.

\begin{proposition}
  \label{prop:equivheight}
  $h_I(E) \asymp h_0(E)$
\end{proposition}

\begin{proof}
  Note that $|X+Y| \le 2 \max\{ |X|, |Y| \}$ and $\max \{ 4|A|, 4|B| \} \ge 4$.
  \begin{align*}
    h_0(E) &= \max \{ h(j), \log |\Delta|, 1 \} \\
    &\le 2 \max \{ h(j), \log \max\{ |4A^3|, |27B^2| \}, 1 \} \\
    %& \le 6 \max \{ h(j), \log \max \{ 4|A|, 4|B| \} + 2 \log 27 \} \\
    & \le 18 \max \{ h(j), \log \max \{ 4|A|, 4|B| \} \} \\
    &= 18 h_I(E)
  \end{align*}
  For the other direction,
  \begin{align*}
    \log \max\{ 4|A|, 4|B| \} 
    &\le \log \max \{ |4A^3|, |27B^2| \} \\
    &\le 2 \log \max \{ |4A^3|, |\Delta| \} \\
    &\le 4 \max \{ h(j), \log |\Delta|, 1\},
  \end{align*}
  from which we conclude that $h_I(E) \le 4 h_0(E)$.
\end{proof}
Lang \cite[Conjecture 5]{MR717593} originally stated the Lang-Hall conjecture in terms of the height
  \[
  h_{LH}(E) = \log \max \{ |A|, |B| \}
  \]
  and the relationship
  \[
  h(P) < C_1 h_{LH}(E) + C_2.
  \]
  One verifies that $h_{LH}(E) \le h_I(E)$, that $h_I(E) \ge 2\log 2$, and much as in Proposition \ref{prop:equivheight}, one has
  \[
  h_0(E) \le  C_1 +  C_2 h_{LH}(E).
  \]
  These facts combine to show the Lang-Hall Conjecture as stated in the introduction is equivalent to that given in \cite{MR717593}.
  
  Lang originally stated the Height Conjecture in terms of $h_{H}(E) = \log |\Delta|$.  Since, $h_H(E) \le h_{0}(E)$, the conjecture with $h_H$ follows from that with $h_0$; the conjecture stated in terms of $h_0$ is actually a strengthened form (see for example \cite[Conjecture VIII.9.9]{MR2514094}).

We are now ready to prove Theorem \ref{thm:patrick-new}.  Ingram's result is as follows. 

\begin{theorem}[{\cite[Theorem 1]{\Ingram}}]
  There is an absolute constant $C$ such that for all minimal elliptic curves $E/\QQ$, and non-torsion points $P \in E(\QQ)$, there is at most one value of $n > CM(P)^{16}$ such that $[n]P$ is integral.  Furthermore, this one value is prime.
  \label{thm:patrick}
\end{theorem}

Recall that $M(P)$ is the smallest integer $n$ such that $[n]P$ has non-singular reduction modulo all primes.  The proof of this result depends on a lemma about valuations of division polynomials, restated here.  Any point $P \in E(\QQ)$, where $E$ is in Weierstrass form, can be written uniquely in the form $\left( \frac{A}{D^2}, \frac{B}{D^3} \right)$ for some $A,B,D \in \ZZ$ with $\gcd(AB,D)=1$ and $D>0$.  We will call $D$ the \emph{denominator} of $P$.

\begin{lemma}[{\cite[Lemma 3]{\Ingram}}]
  Let $E/\QQ$ be an elliptic curve in Weierstrass form, let $P \in E(\QQ)$ be an integral point of infinite order, and let $W_n$ be the associated elliptic divisibility sequence.  Let $D_n$ be the denominator of $[n]P$.  Then, for $n \ge 1$, 
  \label{lemma:patrick}
  \begin{equation*}
    \log D_n \le \log |W_n| \le \log D_n + n^2 M(P)^2 \log |\Delta|.
    \label{eqn:dnwn-patrick}
  \end{equation*}
\end{lemma}
Ingram's proof of Lemma \ref{lemma:patrick} depends on the results of Cheon and Hahn \cite{\CheonHahn} concerning valuations of division polynomials.  With the stronger results of this paper, we can replace $M(P)$ with a constant independent of the curve and point.  The improved lemma is the following.

\begin{lemma}
  Let $E/\QQ$ be an elliptic curve in Weierstrass form, let $P \in E(\QQ)$ be an integral point of infinite order, and let $W_n$ be the associated elliptic divisibility sequence.  Let $D_n$ be the denominator of $[n]P$.  Then, for $n \ge 1$,
  \label{lemma:dnwn}
  \begin{equation*}
    \log D_n \le \log |W_n| \le \log D_n + \frac{n^2}{8} \log |\Delta|.
    \label{eqn:dnwn}
  \end{equation*}
\end{lemma}

\begin{proof}%[Proof of Lemma \ref{lemma:dnwn}]
  Since $P$ is integral, $D_n \mid W_n$, and so we have the first inequality.  To prove the second inequality, we assume $E$ is minimal and look locally at each prime, and show that
  \[
  v(W_n) \le v(D_n) + \frac{n^2}{8}v(\Delta). 
  \]
  
  Write $\phi_n = \phi_n(P)$, $\Psi_n = \Psi_n(P) = W_n$.  The second inequality is a statement about the size of $g_n = \gcd(\phi_n, \Psi_n)$.  Since $E$ is in minimal form, the quantity $g_n$ is supported only on primes for which $P$ has singular reduction \cite[Theorem A]{\AyadS}.  In other words, $v(g_n) = \min\{ v(\phi_n), v(\Psi_n) \} = 0$ for any valuation $v = v_p$ associated to a prime $p$ at which $P$ has non-singular reduction.  In this case, $v(W_n) = v(D_n)$.
  
  We consider the reduction type of $E$ case-by-case.

  Suppose $E$ has multiplicative reduction.  Recall the notation of Section \ref{sec:mult}, especially Definition \ref{defn:aplp} and Theorem \ref{thm:mult}, as well as the elliptic troublemaker sequence $R_n$ of Section \ref{sec:ets}.  Since $P$ has singular reduction, $n_P > 1$.  Since $\ell_P = v(\Delta)$, and $x(1-x) \le 1/4$ for $0 \le x \le 1$,
\[
R_n(a_P,\ell_P) \le \frac{n^2\widehat{a_P}(\ell_P - \widehat{a_P})}{2\ell_P} \le \frac{\ell_Pn^2}{8} = \frac{n^2}{8}v(\Delta).
\]
We know
\[
v(W_n) = R_n(a,\ell) + T_n,
\]
where
\[
T_n = \left\{ \begin{array}{ll}
  v(x([n]P)/y([n]P))  & n_P \mid n \\
  0 & n_P \nmid n \\
\end{array} \right\} = v(D_n)
\]
as in the proof of Theorem \ref{thm:mult}.  We conclude that
\[
v(W_n) - v(D_n) \le \frac{n^2}{8}v(\Delta).
\]

Now suppose that $E$ is of additive potential multiplicative reduction (refer to Section \ref{sec:pot-mult}).  In this case, $v(\Delta) > -v(j) \ge 0$.  We pass to an extension of ramification degree $d$ over which an isomorphism $\phi: E \rightarrow E'$ is defined between $E$ and a minimal $E'$ of multiplicative reduction (guaranteed by Theorem \ref{thm:pot-mult}).  Let $P' = \phi(P)$.  Write $\Delta' := \Delta_{E'}$ and $W_n' := \Psi_n(P',E')$.  Then
by Theorem \ref{thm:pot-mult} and its proof (recall that $v_1$ is a valuation lying above the valuation $v$ of $\QQ_p$ such that $v_1 = dv$ on $\QQ_p$), 
\[
v_1(D_n') \le dv(D_n) + \frac{1}{4}dv(c_4),
\]
as well as
\[
        dv(\Delta) = v_1(\Delta') + 3dv(c_4),
\]
and
\[
dv(W_n) = v_1(W_n') + (n^2-1)dv(c_4)/4.
\]
Recall that (from the standard fact that $j = c_4^3/\Delta$), 
\[
3v(c_4) = v(j) + v(\Delta) > 0.
\]
We obtain
\[
v_1(\Delta') = -dv(j).
\]
Therefore we may compute
\begin{align*}
  dv(W_n) - dv(D_n)
  &\le \frac{n^2-1}{4} dv(c_4)  + v_1(W_n') - v_1(D_n') + \frac{1}{4}  dv(c_4)  \\
  &\le \frac{n^2}{12}\left( dv(j) + dv(\Delta) \right) + \frac{n^2}{8}v_1(\Delta') \\
  &= \frac{n^2}{12} \left( dv(j) + dv(\Delta) \right) - \frac{n^2}{8} dv(j)   \\
  &= -\frac{n^2}{24} dv(j) + \frac{n^2}{12} dv(\Delta) \\
  &\le \frac{n^2}{8} dv(\Delta).
\end{align*}

Suppose that $E$ has additive reduction that resolves to good reduction.  Then, we perform the same sort of computation, but $v_1(W_n') = v_1(D_n')$.  From Theorem \ref{thm:pot-good},
\[
dv(W_n) = (n^2-1)dv(\Delta)/12 + v(W_n')
\]
Then
\begin{align*}
  dv(W_n) - dv(D_n)
  &\le \frac{n^2-1}{12} dv(\Delta)  + v_1(W_n') - v_1(D_n') + \frac{1}{12}  dv(\Delta)  \\
  &= \frac{n^2}{12}dv(\Delta)
\end{align*}
In all cases, we find
\[
v(W_n) - v(D_n) \le \frac{n^2}{8}v(\Delta).
\]
The lemma, for minimal curves, follows by combining this result for all primes.

For a curve which is not minimal, we must apply a change of variables $\phi$ for some $u$ with $v(u) < 0$, where $E'$ is minimal.  We have
\begin{align*}
  v(W_n) - v(D_n) &\le v(W_n') - (n^2-1)v(u) - v(D_n') - v(u) \\
  &\le -n^2 v(u) + \frac{n^2}{8}v(\Delta') \\
  &= -n^2v(u) + \frac{n^2}{8}\left( v(\Delta) + 12 v(u) \right) \\
  &= \frac{n^2}{2} v(u) + \frac{n^2}{8}v(\Delta)  \\
  &\le \frac{n^2}{8}v(\Delta)
\end{align*}
\end{proof}

Ingram's proof of Theorem \ref{thm:patrick} depends upon $M(P)$ in two places:  first, in Lemma \ref{lemma:patrick} that we are replacing with Lemma \ref{lemma:dnwn}; and second, when Ingram bounds the ratio $h(E)/\widehat{h}(P)$ above in \cite[Lemma 5]{\Ingram} (the proof of this lemma uses work of Silverman \cite{MR630588} and Hindry and Silverman \cite{MR948108}).  In our proof, we simply track the dependence on $\widehat{h}(P)/h(E)$ instead of bounding it.

In what follows, we explain the modifications to \cite{\Ingram} necessary to obtain Theorem \ref{thm:patrick-new}.  It should be pointed out that, once Lemma \ref{lemma:dnwn} is in place, the remaining modifications are relatively straightforward and partially follow unpublished notes of Ingram \cite{ProofByAuthority}. However, Ingram's proof spans 11 pages, 95\% of which need not be modified at all.  Therefore, rather than giving the full proof again, we provide details outlining the modifications only.  Most modifications consists of following slight changes in constants.  Where more significant modifications are needed, full proofs of the relevant propositions are given.

Since Ingram considers only short Weierstrass form, he defines \emph{quasi-minimal} to mean a curve with minimal discriminant among short Weierstrass forms with integral coefficients.  Such a curve has a discriminant dividing $6^{12}\mathcal{D}$ where $\mathcal{D}$ is the true minimal discriminant \cite[Proof of Lemma 5]{\Ingram}. 

The first alteration is to a Proposition 4 of \cite{\Ingram}, restated here.

\begin{proposition}[{\cite[Proposition 4]{\Ingram}}]
  \label{prop:patrick-too-patrick}
  Let $E/\QQ$ be an elliptic curve in quasi-minimal Weierstrass form, let $P \in E(\QQ)$ be an integral point of infinite order, and suppose that $[n]P$ is integral for some $n \ge 2$.  Then
  \begin{equation*}
          \widehat{h}(P) \le \log n  + \left( \frac{16}{3}M(P)^2 +2 \right) h(E).
  \end{equation*}
\end{proposition}

We will prove instead

\begin{proposition}
  \label{prop:patrick-too}
  Let $E/\QQ$ be an elliptic curve in quasi-minimal Weierstrass form, let $P \in E(\QQ)$ be an integral point of infinite order, and suppose that $[n]P$ is integral for some $n \ge 2$.  Then
  \begin{equation*}
    \widehat{h}(P) \le \log n  + \frac{16}{3} h(E).
    \label{eqn:patrick-too}
  \end{equation*}
\end{proposition}

\begin{proof}%[Proof of Proposition \ref{prop:patrick-too}]
  The proof is exactly as for \cite[Proposition 4]{\Ingram}, except that we assume $|x(P)| > 240 n^2 \exp(3 h(E)/2)$ and use Lemma \ref{lemma:dnwn} in place of \cite[Lemma 3]{\Ingram}.  The altered proof is included for completeness and because part of it is used later.

  A lemma of David \cite[Lemma 10.1]{MR1385175} states that for $\mathcal{O} \neq Q \in E[n]$,
  \begin{equation*}
    |x(Q)| \le 120n^2\exp(h(E)).	
    \label{eqn:david}
  \end{equation*}
  Suppose that $[n]P$ is an integral point, and suppose that
 \[
 |x(P)| > 240n^2 \exp(4 h(E)/3).
 \]
 Then, $|x(P)| > 2|x(Q)|$, and so $|x(P) - x(Q)| > \frac{1}{2}|x(P)|$, for all $\mathcal{O} \neq Q \in E[n]$.
  From the definition of division polynomials,
  \begin{equation*}
    \Psi_n^2 = n^2 \prod_{Q \in E[n]\setminus \{ \mathcal{O} \}} |x(P) - x(Q)|.
  \end{equation*}
  Therefore,
  \begin{align*}
          2 \log |\Psi_n| &> 2 \log n + (n^2-1)(4 h(E)/3 + 2\log n + \log 120 ) \\
                          &\ge 2\log n + n^2h(E) + (n^2-1)(2\log n + \log 120),
  \end{align*}
  since $\frac43(n^2-1)\ge n^2$ whenever $n \ge 2$.
  On the other hand, as $D_n=1$ (since $[n]P$ is integral), so by Lemma \ref{lemma:dnwn} and the fact that $\log|\Delta| \le 4h(E)$,
  \[
  2 \log |\Psi_n| < n^2 h(E).
  \]
  Combining these two, for $n \ge 2$, we obtain
  \[
  0 \ge 2n^2 \log n + (n^2-1) \log 120
  \]
  which is a contradiction.  Therefore, 
  \begin{equation}
    \label{eqn:needlater}
  |x(P)| \le 240 n^2 \exp(4 h(E)/3).
\end{equation}
  By Silverman \cite[Theorem 1.1]{MR1035944}, for all $P \in E(\QQ)$,
  \[
  \left| \widehat{h}(P) - \frac{1}{2} h(x(P)) \right| < 2 h(E).
  \]
  Since $P$ is integral, $h(x(P)) = \log |x(P)|$.  Therefore,
  \[
  \widehat{h}(P) \le \frac{1}{2}h(x(P)) + 2 h(E)
  \le \frac{1}{2} \log 240 +  \log n + \frac{10}{3} h(E) \le  \log n + \frac{16}{3} h(E)
  \]
  since $h(E) \ge 2\log 2$.

\end{proof}

The remaining modifications mainly involve tracking the differences in various constants throughout.  We will give each statement and its modification.

\begin{lemma}[{\cite[Lemma 6]{\Ingram}}]
        Let $a,b > 0$ be real numbers, and set $f(x) = x^2 - a\log(x) - b$.  Then $f(x) \ge 0$ for $x \ge \max\{ e, a+b \}$.
\end{lemma}

(Note that $e$ is the natural logarithm.)  This is replaced with

\begin{lemma}
        \label{lemma:ab}
        Let $a>0$ be a real number.  Let $f(x) = x^2 - a \log x - a$.  Then $f(x) \ge 0$ for $x \ge \max\{ 9, \sqrt{a\log a} \}$.
\end{lemma}

\begin{proof}
The proof is a straightforward exercise.
\end{proof}

We will use the notation $C_P := h(E)/\widehat{h}(P)$ for simplicity.

\begin{proposition}[{\cite[Proposition 7]{\Ingram}}]
        For all quasi-minimal $E/\QQ$ and non-torsion $P \in E(\QQ)$, there is a constant $c_0$ depending only on $M(P)$, such that if $[n]P$ is integral and $n > c_0$, then $n$ is prime.  Furthermore, we may choose $c_0 = O(M(P)^{16})$, where the implied constant is absolute.
\end{proposition}

\begin{proposition}
        \label{prop:mod7}
        For all quasi-minimal $E/\QQ$ and non-torsion $P \in E(\QQ)$, there is a constant $c_0$ depending only on $C_P$, such that if $[n]P$ is integral and $n > c_0$, then $n$ is prime.  Furthermore, we may choose $c_0 = O(C_P\log(C_P))$, where the implied constant is absolute.
\end{proposition}

\begin{proof}
        The proof mimics that of Ingram.  Supposing that $n$ is composite, put $n = qa$ where $2 \le q \le \sqrt{n}$ is a prime and $q \le a$.  Supposing that $[n]P = [q]([a]P)$ is integral, we may apply Proposition \ref{prop:patrick-too}:
    \[
            a^2 \widehat{h}(P) = \widehat{h}([a]P) \le \log q  + \frac{16}{3} h(E).
    \]
    Therefore, since $q \le a$,
    \[
            a^2 \le \frac{\log a}{\widehat{h}(P)} + \frac{16}{3}\frac{h(E)}{\widehat{h}(P)}.
    \]
    Then, using the fact that $h(E) \ge 1$ and Lemma \ref{lemma:ab}, we have
    \[
            a \le  \max\left\{ 4, \sqrt{\frac{16}{3}\frac{h(E)}{\widehat{h}(P)}\log\left(\frac{16}{3}\frac{h(E)}{\widehat{h}(P)}\right)} \right\}.
    \]
    The stated bound comes from applying this to $n \le a^2$.
\end{proof}

Ingram uses David's explicit lower bounds for linear forms in elliptic logarithms.  Let $\omega$ be the real period of $E$, and let $L_{n,m}(z,\omega) = nz+m\omega$, where $z$ is the principal value of the elliptic logarithm of $P$ and choose $m$ so that $L_{n,m}(z,\omega)$ is the principal value of the elliptic logairthm of $[n]P$.  See \cite[\S 2]{\Ingram} for details.

\begin{lemma}[{\cite[Lemma 9]{\Ingram}}]
        There exist absolute positive constants $c_1$ and $c_2$ such that if $[n]P$ is an integral point and $n>c_2$, then
        \[
                \log | L_{n,m}(z,\omega) | \le -c_1n^2h(E).
        \]
        Furthermore, we may take $c_1^{-1} = O(M(P)^6)$ and $c_2 = O(M(P)^3)$.
\end{lemma}

This is modified to become:

\begin{lemma}
        \label{lemma:mod9}
        There exist absolute positive constants $c_1$ and $c_2$ such that if $[n]P$ is an integral point and $n>c_2$, then
        \[
                \log | L_{n,m}(z,\omega) | \le -c_1n^2h(E).
        \]
        Furthermore, we may take $c_1^{-1} = O(C_P)$ and $c_2 = O(\sqrt{C_P})$.
\end{lemma}

\begin{proof}
        The proof is exactly the same, except that in place of using \cite[Lemma 5]{\Ingram}, we track the dependence on $C_P$.
\end{proof}

\begin{proposition}[{\cite[Proposition 11]{\Ingram}}]
   Let $E/\QQ$ be a quasi-minimal elliptic curve, and let $P \in E(\QQ)$ be a point of infinite order.  There exist positive constants $C_3$ and $c_4$ (depending only on $M(P)$) such that for all $n>c_3$, $[n]P$ integral implies
   \[
           n < c_4 h(E)^{5/2}.
   \]
   Furthermore, we may choose the constants such that $c_3, c_4 = O(M(P)^5 \log^+(M(P))^{3/2})$.
\end{proposition}

This we will replace with

\begin{proposition}
        \label{prop:mod11}
   Let $E/\QQ$ be a quasi-minimal elliptic curve, and let $P \in E(\QQ)$ be a point of infinite order.  There exist positive constants $C_3$ and $c_4$ (depending only on $C_P$) such that for all $n>c_3$, $[n]P$ integral implies
   \[
           n < c_4 h(E)^{5/2}.
   \]
   Furthermore, we may choose the constants such that $c_3 = O(C_P)$ and $c_4 = O(C_P^{1/2})$.
\end{proposition}

\begin{proof}
        The proof is as in Ingram, except that (with reference to the notation there), by Proposition \ref{prop:patrick-too} (in lieu of \cite[Proposition 4]{\Ingram}), it now suffices to take
\[
\log B = \log V_1 \ge 2 \log n + 11 h(E),
\]
and we can use $C' = 10^{46}$.  For curves with $h(E) \ge 2\pi\sqrt{3}$, we can then use the improved constant $c_4 = 10^{24}C_P^{1/2}$ (this depends on Lemma \ref{lemma:mod9}), and choosing any $0 < \epsilon < 1$, we can use 
\[
c_3 = \max\left\{ c_\epsilon, \left( 10^{24}C_P^{1/2} \right)^{\frac{1}{1-\epsilon}} \right\},
\]
where $c_\epsilon$ is a constant such that $\log n < n^{\epsilon/3}$ for all $n > c_\epsilon$.   For example, if $\epsilon = 1/2$, we can take $c_\epsilon = 10^8$.  (Note that Ingram makes an inconsequential error in computing $c_3$.)
\end{proof}

\begin{lemma}[{\cite[Lemma 12]{\Ingram}}]
If $E/\QQ$ is a quasi-minimal elliptic curve, and $P \in E(\QQ)$ is a point of infinite order, then there is a constant $C = O(M(P)^4)$ such that the following holds:  if $z$ is the principal value of the elliptic logarithm of $P$, $\omega$ is the real period of $E$ and $[n]P$ is an integral point, then either $|nz| > \omega/2$ or $n<C$.
\end{lemma}

\begin{lemma}
        \label{lemma:mod12}
        If $E/\QQ$ is a quasi-minimal elliptic curve, and $P \in E(\QQ)$ is a point of infinite order, then there is a constant $C = O(C_P^{1/2})$ such that the following holds:  if $z$ is the principal value of the elliptic logarithm of $P$, $\omega$ is the real period of $E$ and $[n]P$ is an integral point, then either $|nz| > \omega/2$ or $n<C$.
\end{lemma}

\begin{proof}
        The proof is as in Ingram:  we replace Ingram's equation $(10)$ with our \eqref{eqn:needlater}, which does not depend on $M(P)$.  Then we can take $C = \sqrt{5/c_1} = \sqrt{10}C_P^{1/2}$.  The proof depends on the modifications Lemma \ref{lemma:mod9} and Proposition \ref{prop:patrick-too}.
\end{proof}

\begin{proposition}[{\cite[Proposition 13]{\Ingram}}]
        Let $E/\QQ$ be quasi-minimal, and let $P \in E(\QQ)$ be a point of infinite order.  Suppose that $[n_2]P$ and $[n_1]P$ are integral points.  Then there exist constants $c_5 = O(M(P)^6)$ and $c_6 = O(M(P)^{16})$, such that
        \[
                n_1^2 h(E) \le c_5 \log n_2
        \]
        whenever $n_1, n_2 > c_6$.
\end{proposition}

We replace this with 

\begin{proposition}
        \label{prop:mod13}
        Let $E/\QQ$ be quasi-minimal, and let $P \in E(\QQ)$ be a point of infinite order.  Suppose that $[n_2]P$ and $[n_1]P$ are integral points.  Then there exist constants $c_5 = O(C_P)$ and $c_6 = O(C_P^{1/2})$, such that
        \[
                n_1^2 h(E) \le c_5 \log n_2
        \]
        whenever $n_1, n_2 > c_6$.
\end{proposition}

\begin{proof}
        The proof is as in Ingram; we use $c_5 = 2/c_1 = 4C_P$ and $c_6 = \max\{ c_0, C, K \}$ where $K$ is an absolute constant.   The proof relies on Lemmas \ref{lemma:mod9} and \ref{lemma:mod12}, and Proposition \ref{prop:mod7}.
\end{proof}

For clarity, we now present the proof of Theorem \ref{thm:patrick-new}, following \cite[Theorem 1]{\Ingram}, but using the modified propositions and lemmas.

\begin{proof}[Proof of Theorem \ref{thm:patrick-new}]
  Let $E/\QQ$ be a quasi-minimal elliptic curve with an integral point $P \in E(\QQ)$ of infinite order (if $P$ were not integral, it would not have any integral multiples).
Let $C_0 = \max\{ c_0, c_3, c_6, c_7 \}$, where
\[
c_7 = \sqrt{c_5 \log c_4} 
\]
If $[n_1]P$ and $[n_2]P$ are integral and $C_0 < n_1, n_2$, then by Propositions \ref{prop:mod11} and \ref{prop:mod13}, we have
\[
n_1^2 h(E) \le c_5 \log n_2, \quad \mbox{and} \quad n_2 \le c_4 h(E)^{5/2}.
\]
Combining these, we have
\begin{equation}
        \label{eqn:replaceconstants}
h(E) \le \frac{5c_5}{2n_1^2} \log h(E) + \frac{c_5}{n_1^2} \log c_4 .
\end{equation}
Recall that
\[
c_5 = O(C_P),  \quad c_4 = O(C_P)^{1/2}
\]
and since $n_1 > c_6 \ge O(C_P^{1/2})$, the first constant in \eqref{eqn:replaceconstants} can be replaced with an absolute constant, and since $n_1 > c_7$, the second can also.  We therefore obtain an absolute upper bound
\[
h(E) \le N.
\]
On those $E$ with $h(E) > N$, there can be at most one $n > C_0$ such that $[n]P$ is integral.  Let
\[
C_0' = \sup_{h(E) \le N} \{ n: [n]P \mbox{ is integral for some }P \in E(\QQ) \}.
\]
The set of $h(E) \le N$ is finite and can be effectively computed, if $N$ is known.  Letting $C = \max\{ C_0, C_0' \}$, and we have shown that there is at most one value of $n > C$ such that $[n]P$ is integral. 

It remains to simplify the constant $C_0$.  Considered as a function of $x = C_P$, it is of the form
\[
C_0 = \max\{ K_0, K_1 x(\log x), K_2 x, K_3 x^{\frac{1}{2}}, K_4 x^{\frac{1}{2}} \left( \log x \right)^{\frac{1}{2}} \},
\]
where the $K_i$ are absolute constants.  If we increase the constant $K_0$ sufficiently, then since $x(\log x)$ grows fastest (as $x$ increases) among all the functions (which are all eventually increasing), we may replace $C_0$ with
\[
C_0 = \max\{ K_0', K_1 x (\log x) \}.
\]
This proves the theorem.
\end{proof}

\section{Other connections and applications}
\label{sec:applications}

\subsection{Growth rates of valuations}
\label{subsec:growth}

The main theorems of this paper give growth rates of $v(W_n)$.  Cheon and Hahn find that for a non-torsion point over a number field with singular reduction, the growth rate is quadratic \cite{\CheonHahn}.  Everest and Ward give more precise growth information in \cite[Theorem 3]{MR1800354}, which says that for any $E$ in minimal Weierstrass form and $P$ of singular reduction,
\[
\log | \Psi_n(P) |_v =  \left( \log | \Delta_E |_v /12 + \lambda_v(P) \right) n^2 + O(n^C),
\]
where $C < 2$ and may depend on $P$ (here, $\lambda_v(P)$ is a canonical local height; see \cite[\S VI.2]{\Siltwo}).

The results of this paper allow us to improve this estimate.  For a point of singular reduction on a curve of additive reduction, the coefficient of $n^2$ depends on the behaviour of the point when the field is extended to resolve the additive reduction; see Theorems \ref{thm:pot-good} and \ref{thm:pot-mult}, and the examples of Section \ref{sec:examples}.  

For multiplicative reduction, the constant is more easily stated.  If $E$ is in minimal form, then from Theorem \ref{thm:mult}, Proposition \ref{prop:rn}\ref{item:growth}, and Proposition \ref{prop:snprop}\ref{item:sngrowth},
      \[
      v(W_n) =  \left( \frac{ a_P (\ell_P - a_P) }{2 \ell_P } \right) n^2 + O(\log n).
      \]
where the meaning of $a_P$ and $\ell_P$ is given in Definition \ref{defn:aplp}.  In particular, 
\[
0 < a_P \le \ell_P = v(\Delta_E).
\]
Using \cite[Theorem VI.4.2(b)]{\Siltwo}, it is immediate to verify that this constant is in agreement with Everest and Ward's.   

In all cases (i.e. all types of bad reduction), our theorem improves Everest and Ward's result.  We have
\begin{theorem}
  Let $K$ be a $p$-adic field, with valuation $v$, and residue field of size $N_K$.  Let $E$ be a minimal elliptic curve over $K$ and let $P \in E(K)$ be a point with singular reduction.  Let $W_n$ be the associated elliptic divisibility sequence.  Then
  \[
  v(W_n) = \left( \frac{\lambda_v(P)}{\log|N_K|} + \frac{v(\Delta_E)}{12} \right) n^2 + O(\log n).
  \]
  \label{thm:growth}
\end{theorem}

\begin{proof}
  From \cite[Theorem 3]{MR1800354}, all that remains is to show that the error term is correct.  This follows from Propositions \ref{prop:snprop}\ref{item:sngrowth} and \ref{prop:rn}\ref{item:growth} and Theorems \ref{thm:change-to-minimal}, \ref{thm:non-sing}, \ref{thm:pot-good}, \ref{thm:mult} and \ref{thm:pot-mult}.
\end{proof}

\subsection{Torsion points and Tate normal form}
\label{subsec:gezerbizim}

Gezer and Bizim use Tate's normal form for an elliptic curve with an $N$-torsion point to obtain general formul{\ae} for EDS of rank $N$ \cite{GezerBizim}.  For example, the general form of a rank $7$ EDS is
\[
1, -\alpha^2 (\alpha - 1), -\alpha^6 (\alpha - 1)^3 , \alpha^{11} (\alpha - 1)^6 \ldots
\]
They go on to give the general term as
\[
W_n = \epsilon \alpha^{(5n^2 - p)/7} (\alpha-1)^{(3n^2-q)/7},
\]
where
\[
\epsilon = \left\{ \begin{array}{ll}
  +1 & \mbox{if }n \equiv 1,4,5 \pmod{7} \\
  -1 & \mbox{if }n \equiv 2,3,6 \pmod{7} \\
\end{array} \right.
\]
\[
p = \left\{ \begin{array}{ll}
  5 & \mbox{if }n \equiv 1,6 \pmod{7} \\
  6 & \mbox{if }n \equiv 2,5 \pmod{7} \\
  3 & \mbox{if }n \equiv 3,4 \pmod{7} \\
\end{array} \right. , \quad
q = \left\{ \begin{array}{ll}
  3 & \mbox{if }n \equiv 1,6 \pmod{7} \\
  5 & \mbox{if }n \equiv 2,5 \pmod{7} \\
  6 & \mbox{if }n \equiv 3,4 \pmod{7} \\
\end{array} \right. .
\]
We can now restate this general term as
\[
W_n = \epsilon \alpha^{R_n(2,7)} (\alpha-1)^{R_n(1,7)},
\]
where $\epsilon$ is as above.

\section{Examples}
\label{sec:examples}

The examples in this section illustrate the main theorems of the paper describing $v(W_n)$, both the usual and unusual.

\begin{example}
  \label{example:mult}
  This example demonstrates non-singular reduction fitting the hypotheses of Corollary \ref{cor:non-sing-nice}, as well as singular reduction on a curve of multiplicative reduction.  Consider the elliptic curve in minimal Weierstrass form and point
  \[
  E: y^2 + xy = x^3 + x^2 - 1652x + 25168, \quad P = (24,-4),
  \]
  having $j =-2^{-8} \cdot 7^{-2} \cdot 11^3 \cdot 89^{-1} \cdot 7211^3$, $\Delta = - 2^8 \cdot 7^2 \cdot 89$, and $c_4 =11 \cdot 7211$.

  The curve has good reduction at $p=3$.  The point $P$ reduces to a point of order $5$.  By Corollary \ref{cor:non-sing-nice},  
   \[
   v_3(W_n) = \left\{ \begin{array}{ll} v(W_5) + v(n/5) & 5 \mid n \\
    0 & 5 \nmid n \end{array} \right. .
  \]

  The curve has multiplicative reduction at $p=7$.  The point $P$ reduces to a non-singular point of order $6$.  By Corollary \ref{cor:non-sing-nice},
  \[
  v_7(W_n) = 
  \left\{ \begin{array}{ll} v(W_6) + v(n/6) & 6 \mid n \\
    0 & 6 \nmid n \end{array} \right. .
  \]

  The curve has multiplicative reduction at $p=2$.  The point $P$ reduces to the singular point.  The smallest multiple of $P$ reducing to the identity is $[6]P = (4719/196, -56771/2744)$.  In Theorem \ref{thm:mult}, $\ell_P = v_2(\Delta)=8$.  Since $[2]P$ has non-singular reduction, $P$ reduces to the component of $E(\QQ_2)/E_0(\QQ_2)$ having order $2$, i.e. $a_P = 4$.  Using the notations of Lemma \ref{lemma:vals-form}, $h=0$ and $b=p$ by Theorem \ref{thm:mult}.  Also, $s_P = v_2(\Theta([6]P))=1$ and so $j=0$.  Furthermore, $w_P = v_2(\Theta([12]P)/\Theta([6]P)^2) = 3 - 2 = 1$.  Therefore,
  \[
  v_2(W_n) = R_n(4,8) +
  \left\{ \begin{array}{ll} 2 + v(n/6) & v_6(n) > 1 \\
    1 & v_6(n) =1  \\
    0 & 6 \nmid n \end{array} \right. .
  \]

  The EDS associated to $E$ and $P$ is 
\[
1, \; 2^4,\;  2^8,\;  2^{16},\;  2^{24} \cdot 3 \cdot 5,\; 2^{37} \cdot 7,\;  - 2^{48}, - 2^{64} \cdot 211, \;
- 2^{80} \cdot 23 \cdot 137, \ldots
\]
with valuations, agreeing with the formul{\ae} above, of
\begin{align*}
v_2(W_n):& \; 0, 4, 8, 16, 24, 37, 48, 64, 80, 100, 120, 147, 168, 196, 224, 256,
288, \\ &\;\;\; 325, 360, 400, 440, 484, 528, 580, 624, 676, 728, 784, 840, 901,
960, \ldots \\
v_3(W_n):& \; 0, 0, 0, 0, 1, 0, 0, 0, 0, 1, 0, 0, 0, 0, 2, 0, 0, 0, 0, 1, 0, 0, 0, 0,
1, 0, 0, 0, 0, 2, \\ &\;\;\; 0, 0, 0, 0, 1, 0, 0, 0, 0, 1, 0, 0, 0, 0, 3, 0, 0, 0,
0, 1, 0, 0, 0, 0, 1,  \ldots \\
v_7(W_n):& \; 0, 0, 0, 0, 0, 1, 0, 0, 0, 0, 0, 1, 0, 0, 0, 0, 0, 1, 0, 0, 0, 0, 0, 1,
0, 0, 0, 0, 0, 1, \\ &\;\;\; 0, 0, 0, 0, 0, 1, 0, 0, 0, 0, 0, 2, 0, 0, 0, 0, 0, 1,
0, 0, 0, 0, 0, 1,  \ldots
\end{align*}

\end{example}

\begin{example}
  \label{example:identity}
  This example describes a point which reduces to the identity, as well as a point of singular reduction on a curve of additive, potential good reduction.  Consider the elliptic curve in minimal Weierstrass form and point
  \[
  E: y^2 = x^3 + 2471x+1, \quad P = \left( \frac{1}{5^2}, \frac{1249}{5^3} \right),
  \]
  having $j=2^8 \cdot 3^3 \cdot 7^3\cdot  353^3 \cdot 60350132471^{-1}$, $\Delta = - 2^4 \cdot 60350132471$, and $c_4 =- 2^4 \cdot 3 \cdot 7 \cdot 353$.  
  
  This curve has good reduction at $p=5$, but $P$ reduces to the identity.  We are in the case of Theorem \ref{thm:non-sing}, and $v_5(x(P))/2 = -1$.  We have $n_P=1$.  The formal group for this elliptic curve has
  \[
  [5]T = 5T - 3083808T^5 - 33480T^7 + 1574818510720T^9 + O(T^{10}).
  \]
  Therefore, in Lemma \ref{lemma:vals-form}, $b=5$, and $j=h=w=0$.  Therefore, from Theorem \ref{thm:non-sing}, we expect 
  \[
  v_5(W_n) = -n^2 + 1 + v_5(n).
  \]

  At $p=2$, this curve has additive reduction, but potential good reduction.  If we extend $\QQ_2$ by a cube root $\pi$ of $2$ (an extension of ramification degree $3$), then $E$ obtains good reduction.  The change of coordinates is
  \[
  y' = \pi^{-3}(y+x+1), \quad x' = \pi^{-2}(x+1),
  \]
  and the new curve (now in minimal Weierstrass form) and point are
  \[
  E':  y^2 + \pi^2 xy + y = x^3 + \pi x^3 + 618 \pi^2 x + 618, \quad P' = \left( -\frac{12\pi}{5^2}, \frac{622}{5^3} \right).
  \]
  The point $P'$ reduces modulo $\pi$ to the point $(0,0)$ of order $3$ on the reduced curve $y^2 + y = x^3$ over $\FF_2$.  Applying Theorem \ref{thm:non-sing} to $W_n'$, the elliptic divisibility sequence for $E'$ and $P'$, we have $n_P = 3$, $s_P = v_{\pi}(\Theta([3]P)) = 1$.  The formal group for $E'$ has
  \[
  [2]T = 2T - \pi^2 T^2 - 2\pi^2 T^3 + O(T^{10})
  \]
  so that $b=h=2$, $v(p)=3$ and so $j=0$ in Lemma \ref{lemma:vals-form}.  We have $w_P = v_{\pi}(\Theta([6]P)/\Theta([3]P)^2) - h_P = 4 - 2\cdot 1 - 2 = 0$.  Therefore,
  \[
  v_{\pi}(W_n') = \left\{ \begin{array}{ll}
    1 + v_{\pi}(n/3) & 3 \mid n \\
    0 & 3 \nmid n
  \end{array} \right. .
  \]
  By Theorem \ref{thm:pot-good}, we have
  \[
  3 v_2(W_n) = (n^2-1) + v_{\pi}(W_n') = n^2 +  
  \left\{ \begin{array}{ll}
     3v_2(n/3) & 3 \mid n \\
    -1 & 3 \nmid n
  \end{array} \right. .
 \]
  
  The elliptic divisibility sequence for $E$ and $P$ begins
  \begin{multline*}
  1, \; \; 2 \cdot 5^{-3} \cdot 1249, \;\; -1 \cdot 2^3 \cdot 5^{-8} \cdot 298135585859, \\ -1 \cdot 2^5 \cdot 5^{-15}\cdot 1249 \cdot 460436473420870703, \ldots
\end{multline*}
and has valuations, agreeing with the formul{\ae} above, of
\begin{align*}
v_2(W_n):& \; 0, 1, 3, 5, 8, 13, 16, 21, 27, 33, 40, 50, 56, 65, 75, 85, 96, 109,
120, 133, \\ &\;\;\; 147, 161, 176, 195, 208, 225, 243, 261, 280, 301, 320, 341,
363, 385, \\ &\;\;\; 408, 434, 456, 481, 507, 533, 560, 589, 616, 645, 675, 705,
736, 772, \ldots \\
v_5(W_n):& \;   0, -3, -8, -15, -23, -35, -48, -63, -80, -98, -120, -143, -168, \\ 
&\;\;\; -195, -223, -255, -288, -323, -360, -398, -440, -483, -528,  \\
&\;\;\; -575, -622, -675, -728, -783, -840, -898, -960, -1023, \ldots 
\end{align*}

\end{example}

\begin{example}
  \label{example:2-torsion}
  This example showcases a non-integral torsion point.  Consider the elliptic curve, in minimal Weierstrass form, and point
  \[
  E: y^2 + xy + y = x^3 + x^2 - 135x - 660, \quad P = (-29/4, 25/8)
  \]
  The discriminant is $\Delta = 3^8 \cdot 5^2$.  The point $P$ reduces modulo $2$ to the identity.  Therefore, Theorem \ref{thm:non-sing} applies, with $n_P=1$.  We have $v(x(P)) = -2$, so $s_P = 1$, and in the formal group, we have
  \[
  [2]T = 2T - T^2 - 2T^3 - 6T^4 + O(T^5)
  \]
  so that $b_P=2$, $h_P=0$, $j=0$ and $w_P = \infty$ (since $[2]P$ is the identity on $E$).  We obtain
  \[
  v(W_n) = -n^2 + \left\{ \begin{array}{ll}
    \infty & 2 \mid n \\
    1 & 2 \nmid n
  \end{array} \right. .
  \]
  The elliptic divisibility sequence for $E$ and $P$ begins
  \begin{equation*}
    1, \; 0, \; -2^{-8} \cdot 3^8, \; 0, \; 2^{-24} \cdot 3^{24}, \; 0, \; - 2^{-48} \cdot 3^{48}, \ldots
\end{equation*}
and has valuations, agreeing with the formula above, of
\begin{align*}
v_2(W_n):& \; 
0, \infty, -8, \infty, -24, \infty, -48, \infty,
-80, \infty, -120, \infty, -168, \infty, -224, \ldots
\end{align*}

\end{example}

\begin{example}
  \label{example:pot-mult}
  This example illustrates singular reduction on a curve of potential multiplicative reduction.  Consider the curve, in minimal Weierstrass form, and point
  \[
  E: y^2 + 49y = x^3 + 14x^2 - 312352901x +2123335052286, \quad P = (10206, 1176).
  \]

  Modulo $7$, the point $P$ reduces to the cusp $(0,0)$ on the reduced curve, $y^2 = x^3$ (additive reduction).  If we pass to a ramified quadratic extension of $\QQ_7$, say by adjoining a square root $\pi$ of $7$, then the change of coordinates $x' = \pi^{-2}x, y' = \pi^{-3}y$ gives a minimal Weierstrass equation,
  \[
  E': y^2 + \pi y = x^3 + 2x^2 -6374549x + 6190481202, \quad P' = (1458, 24\pi). 
  \]
  having $v_\pi(j) = -10$, $v_\pi(\Delta)=10$, $v_{\pi}(c_4) = 0$.  Therefore, this curve has multiplicative reduction.  The point $P'$ reduces to the node $(2,0)$ on the reduced curve $y^2 = x^3 + 2x^2 + x + 3$.  We have $\ell_P = v_\pi(\Delta) = 10$.  The points $P$ and $[3]P$ reduce to the node, while $[2]P$ reduces to the point $(1,0)$ of order $2$; $[4]P$ reduces to the identity.  Therefore $n_P = 4$.  By Theorem \ref{thm:mult}\ref{item:mult-ns}, $a_P=5$.  Alternatively, $a_P$ must have order $2$ in $\ZZ/\ell_P\ZZ$, so it must be $a_P=5$.  Using the notations of Lemma \ref{lemma:vals-form}, $b=p$ and $h=0$ by Theorem \ref{thm:mult}.  We can compute $s_P = v(\Theta([4]P)) = 1$, which tells us that $j=0$ and $w_P=0$.  Gathering together these parameters, we obtain the sequence of valuations for $W_n'$, the EDS associated to $E'$ and $P'$:
  \[
  v_\pi(W_n') = 
  R_n(5,10) + \left\{ \begin{array}{ll} 1 + v_\pi(n/7) & 7 \mid n \\ 0 & 7 \nmid n . \end{array} \right. .
    \]
  By Theorem \ref{thm:pot-mult}, we have
  \[
  2v_7(W_n) = (n^2-1) + 
  R_n(5,10) + \left\{ \begin{array}{ll} 1 + 2v_7(n/7) & 7 \mid n \\ 0 & 7 \nmid n . \end{array} \right. .
    \]
    The elliptic divisibility sequence $W_n$ begins
    \[
    1,\; 7^4,\; 7^9,\; 7^{18},\; 2 \cdot 3^2 \cdot 7^{27} \cdot 19, \ldots
    \]
and has valuations, agreeing with the formula above, of
\begin{align*}
v_7(W_n):& \;0, 4, 9, 18, 27, 40, 54, 72, 90, 112, 135, 162, 189, 220, 252, 288,
324, 364, \\ &\;\;\; 405, 450, 495, 544, 594, 648, 702, 760, 819, 883, 945, 1012,
1080, \\ &\;\;\; 1152, 1224, 1300, 1377, 1458, 1539, 1624, 1710, 1800, 1890, 1984, \ldots
\end{align*}

\end{example}

\begin{example}
  \label{example:sing-pot-good}
  This example of potential good reduction exhibits very unusual, complicated behaviour.  In particular, we have an example with $j \neq 0$ in Lemma \ref{lemma:vals-form}.  Let $K = \QQ_2$, and $R$ be its ring of integers.  Let $\alpha = \sqrt{17} \in R^*$.  Then the curve
\[
E:  y^2 = x^3 + \alpha x + \alpha + 2
\]
has $j = 2^8 + 2^{10} + 2^{14} + 2^{17} + O(2^{19}) \in R$, so $E$ has potential good reduction.  It is a minimal Weierstrass equation since $v_2(\Delta) = 4 < 12$.  It has additive reduction, since $v_2(c_4) = 4 > 0$.

The reduced curve over $\mathbb{F}_2$ is 
\[
\widetilde{E}:  y^2 = x^3 + x + 1
\]
which has a cusp at $(1,1)$.

Let $\beta^2 = (-17)^3 + \alpha (-17) + \alpha  + 2$.  Then $\beta \in R^*$.  Let $P = (-17, \beta) \in E(K)$.  The point $P$ has singular reduction to the cusp $(1,1)$, but $[2]P$ reduces to the non-singular two-torsion point $(0,1)$.

We have to pass to a ramified extension $L/\QQ_2$ to obtain good reduction for $E$, which will guarantee non-singular reduction for $P$.  It will suffice to change coordinates to Deuring normal form, 
\[
E_D:  y^2 + a xy + y = x^3.
\]
The change of coordinates required is
\[
x = u^2x' + r, \quad y = u^3y' + u^2sx' + t
\]
where $q=s-1$ is a root of the irreducible polynomial
\[
p(x) = (x+1)^8 + 18 \alpha (x+1)^4 + 108(\alpha+2) (x+1)^2 - 27 \alpha^2,
\]
whose constant term, $a_0 = 217 + 126\alpha - 27\alpha^2$, has valuation $v_2(a_0) = 2$.  Therefore $q$ is not a uniformizer (since the polynomial is not Eisenstein), but it has positive valuation.  Let $N = \QQ_2(s)$ have valuation $v_N = dv_2$ where $d$ is the ramification degree of $N$ over $\QQ_2$.  Since $v_N(p(q)-q^8 - a_0) > v_N(a_0) = 2d$ (all the intermediate terms of the polynomial are divisible by $4x$), we find that $8v_N(q) = 2d$, i.e. $v_N(q) = d/4$.  Hence the extension is totally ramified ($d=8$), and $v_N(s-1) = 2$.

We also have $u^3 = (\alpha + s^4/3)/s \in \QQ_2(s)$.  We can compute the valuation of $3\alpha + s^4$ in $N$ as follows.  We have
\[
3\alpha = (1+2)(1+2^3 + O(2^5)) = 1 + 2 + 2^3 + O(2^4).
\]
Meanwhile, 
\[
s^4 = 1 + 4q + 6q^2 + 4q^3 + q^4.
\]
So, $v(3\alpha + s^4) = 8$.  Therefore, $v_N(u^3) = 8$.  Hence, $u$ generates a totally ramified extension $L$ of degree $3$ over $N = \QQ_2(s)$.  Therefore, $[L:\QQ_2] = 24$.

We have $r = s^2/3$ and $t = u^{3}/2$.  Finally, $a=2s/u$.  

The Deuring normal form is a minimal Weierstrass equation of good reduction.  (We could also verify that $v_L(u) = 8$ since $v_L(\Delta_E) = 24 v_2(\Delta_E) = 96$ and so $0 = v_L(\Delta_{E_D}) = v_L(u^{-12}\Delta_E) = v_L(\Delta_E) - 12v_L(u) = 96 - 12v_L(u)$.  We also find that $v_L(a) = 24 + v_L(s) - v_L(u) = 16$, so $a$ is an integer, which confirms that $E_D$ has good reduction.)

Let $\phi: E \rightarrow E_D$ represent the change of coordinates to Deuring normal form.  Then,  
\[
v(x(\phi(P)))  = v_L(17 + s^2/3) - 2v_L(u) = 12 - 16 = -4.
\]

The EDS $W_n'$ associated to the curve $E_D$ and point $P_D = \phi(P)$ satisfies
\[
v_L(W'_n) = 24 v_2(W_n) - 8(n^2-1)
\]
and is associated to a point of non-singular reduction.  In fact, $P_D$ reduces to the point at infinity and has $v(x(P_D)) = -4$.  Thus, the sequence
\[
v_L(W'_n) + 2n^2 = 24 v_2(W_n) - 8(n^2-1) + 2n^2 
\]
must be of the form $S_n(p,t,d,h,s,w)$ as in Lemma \ref{lemma:vals-form}.  Multiplication-by-$2$ in the formal group for $E_D$ begins
\[
[2]T = 2T - aT^2 + (1+ a)T^4  \ldots,
\]
and since $v_L(a) = 16$, we get $b=4$ in Lemma \ref{lemma:vals-form}, and so we have $t=2$, $c=2$, $j=1$ and $w=6$ in Definition \ref{defn:sn}.  The sequence $v_L(W'_n) - 2n^2 = S_n(2,2,24,0,2,6)$ is
\begin{align*}
&2, 8, 2, 32, 2, 8, 2, 56, 2, 8, 2, 32, 2, 8, 2, 80, 2, 8, 2, 32, 2, 8, 2,
56, \\ &2, 8, 2, 32, 2, 8, 2, 104, 2, 8, 2, 32, 2, 8, 2, 56, 2, 8, 2, 32, 2,
8, 2, 80, \\ & 2, 8, 2, 32, 2, 8, 2, 56, 2, 8, 2, 32, 2, 8, 2, 128, 2, 8, 2,
32, 2, 8, 2, 56, \\ & 2, 8, 2, 32, 2, 8, 2, 80, 2, 8, 2, 32, 2, 8, 2, 56, 2,
8, 2, 32, 2, 8, 2, 104, 2, 8, 2, \ldots
\end{align*}
Now let us verify this directly.  The first few terms of the elliptic divisibility sequence associated to $E$ and $P$ are
\begin{align*}
&1, \; 2\beta, \; -\alpha^2 + 1530\alpha
+ 250155, \\ 
&\quad \quad -4\beta\alpha^3 - 5540\beta\alpha^2 + 1277796\beta\alpha + 95764068\beta, \; \ldots
\end{align*}
or
\begin{align*}
& 1,\;  2 + 2^6 + 2^9 + 2^{11} + 2^{12} + 2^{13} + 2^{16} + 2^{18} + O(2^{19}), \\
& \quad \quad  2^2 +
2^5 + 2^7 + 2^8 + 2^{10} + 2^{11} + 2^{12} + 2^{14} + 2^{17} + O(2^{19}), \\
& \quad \quad 2^5
+ 2^7 + 2^9 + 2^{10} + 2^{12} + 2^{13} + 2^{15} + 2^{16} + 2^{18} + O(2^{19}) \; \ldots
\end{align*}
The valuations $v_2(W_n)$ are,
\begin{align*}
& 0,\; { 1},\; 2,\; { 5},\; 6,\; { 9},\; 12,\; { 18},\; 20,\; { 25},\; 30,\; { 37},\; 42,\; { 49}, \\
& \quad\quad 56,\; { 67},\; 72, \; { 81},\; 90,\; { 101},\; 110,\; { 121},\; 132,\;{ 146}, \;156, \;\ldots 
\end{align*}
These are exactly equal to
\[
\frac{1}{24} \left( 8(n^2-1) - 2n^2 + S_n(2,2,24,0,2,6) \right).
\]

\end{example}

%%%%%%%%%%%%%%%%%%%%%%%%%%%%%%%%%%%%%%%%%%%%%%%%%%%%%%%%
%   The Bibliography
%%%%%%%%%%%%%%%%%%%%%%%%%%%%%%%%%%%%%%%%%%%%%%%%%%%%%%%%

\bibliographystyle{amsplain}
\bibliography{EDS}
\end{document}